\documentclass[11pt]{amsart}
\usepackage{amsmath,amssymb,amscd,graphicx, color,tensor,hyperref,ulem}
\newcommand{\R}{{ \mathbb{R}  }}
 
\newcommand{\bke}[1]{\left( #1 \right)}
\newcommand{\bkt}[1]{\left[ #1 \right]}
\newcommand{\bket}[1]{\left\{ #1 \right\}}
\newcommand{\norm}[1]{\left\Vert #1 \right\Vert}

\newcommand{\calK}{{\mathcal K}}

\newcommand{\calI}{{ \mathcal I }}

\newcommand{\calN}{{\mathcal N}}

\newcommand{\calS}{{ \mathcal S  }}
\newcommand{\bbr}{\mathbb R}
\newcommand{\pa}{\partial}
\newcommand{\vp}{\varphi}
\newcommand{\om}{{ \omega  }}
\newcommand{\Om}{{ \Omega  }}
\newcommand {\la}{\lambda}
\newcommand{\na}{\nabla}
\newcommand {\ga}{\gamma}
\newcommand {\al}{\alpha}
\newcommand {\be}{\beta}

\newcommand{\dds}{\frac{\pa}{\pa s}}
\newcommand{\ddp}{\frac{\pa }{\pa \vp}}
\newcommand{\ddt}{\frac{d}{dt}}

\newcommand{\Del}{\Delta}

\newcommand{\br}{\mathbf r}
\newcommand{\ep}{\epsilon}

\newcommand{\iom}{{L^{\infty}(\Om)}}

\newcommand{\into}{\int_{\Om}}
\newcommand{\tzp}{|\tilde z|^{p}}
\newcommand{\tzpp}{|\tilde  z|^{p-1}}
\newcommand{\tz}{\tilde z}
\newcommand{\emu}{e^{\mu t}}

\begin{document}
\bibliographystyle{plain}

\newtheorem{defn}{Definition}
\newtheorem{lemma}{Lemma}
\newtheorem{proposition}{Proposition}
\newtheorem{theorem}{Theorem}
\newtheorem{assumption}{Assumption}
\newtheorem{cor}{Corollary}
\newtheorem{remark}{Remark}
\numberwithin{equation}{section}
\newenvironment{pfthm1}{{\par\noindent\bf
           Proof of Theorem~\ref{GETHM}. }}{\hfill\fbox{}\par\vspace{.2cm}}
\newenvironment{pfthm2}{{\par\noindent\bf
           Proof of Theorem~\ref{GETHM1}. }}{\hfill\fbox{}\par\vspace{.2cm}}


\title[Nonlocal adhesion models for two cancer cell phenotypes]{
Nonlocal adhesion models for two cancer cell phenotypes in a multidimensional bounded domain}


%
%
%

\author{Jaewook Ahn}%
\address{Department of Mathematics, Dongguk University, Seoul, 04620 Republic of Korea}%
\email{ jaewookahn@dgu.ac.kr}
\author{Myeongju Chae}%
\address{Schoo of applied mathematics and computer engineering, Hankyong University, Anseong, 17579 Republic of Korea}%
\email{mchae@hknu.ac.kr}
\author{Jihoon Lee}%
\address{Department of Mathematics, Chung-Ang University, Seoul, 06974 Republic of Korea}%
\email{jhleepde@cau.ac.kr}

\begin{abstract}
Cell-cell adhesion is an inherently nonlocal phenomenon. Numerous partial differential equation models with nonlocal term have been recently presented to describe this phenomenon, yet the mathematical properties of nonlocal adhesion model are not well understood. Here we consider a model with two kinds of nonlocal cell-cell adhesion, satisfying no-flux conditions in a multidimensional bounded domain. We show global-in-time well-posedness of the solution to this model and obtain the uniform boundedness of solution.
  \end{abstract}
  \maketitle

\section{Introduction}
In this paper we consider  nonlocal cell-cell adhesion models on a multidimensional bounded domain. Cellular adhesion is a fundamental feature of multicellular organisms with relevance for e.g., embryogenesis, wound healing and self-organization.  The regulation of cellular adhesion is critical, too, for  cancer growth and spread.  \\  
\indent
Mathematical models of cancer invasion have been widely studied for decades; the proliferation and movement of tumour cells are governed by  random diffusion, aggregations and reactions which are often of logistic type \cite{ACN,BLM,CA,CL,PB}.  In most cancer migration models, aggregations arise from the cell-cell and cell-ECM (extracellular matrix) adhesions, movement of cells in response to stimulii of diffusing chemicals (chemotaxis), and movement of cells in response to non-diffusing environmental factor such as ECM (haptotaxis). There have been numerous mathematical analysis studies of parabolic  or parabolic-ellipic-ODE system consisting of diffusion, aggregation and reaction for describing tumour dynamics; they addressed local and global well-posedness, blow-up, asymptotic behavior etc. on the whole space or space with boundary. {See \cite{LMR10,MR08,SSW14,TW11,TM201,TM20,TW08} and references therein}. 
\\
\indent
In this paper we study a two-populations cancer model \eqref{MODEL01} set in a multidimensional bounded domain and focused on nonlocal cell-cell adhesion. The first successful continuum model of cellular adhesion was proposed by Armstrong et al. in \cite{APS06}, where
 they introduced a nonlocal integral term  to describe the adhesive forces between cells. The basic single population model on the real line, for instance, is derived to be 
\begin{align}\label{Single} \pa_t u = u_{xx}- \al \left( u(x,t) \int_{-R} ^{R} g(u(x+r)) \om(r) dr \right)_x 
\end{align}
where $\al$ is the strength of cell adhesion, $g(u)$ describes the nature of the adhesion force, $\om(r)$ is a function describing the direction and magnitude of the adhesion force, and $R$  sensing radius of a cell.
In the same paper a two-dimensional version of \eqref{Single} involving two cell populations was also proposed with numerical simulations supporting the active cell-sorting process from the  randomly distributed mixture. The model was the first to reproduce Steinberg's cell segregating experiment \cite{Stein}, which is a classical result in developmental biology. 
\\
\indent
 The regulation of cellular adhesion is critical in cancer formations as well. Naturally  \eqref{Single} has been extensively used  to model cancer cell invasion and developmental processes \cite{ANCH,APS09,BCE17,DTGC,GC08,GP10,PBSG,PAS10,SGAP}. Cauchy problems were studied in \cite{ANCH,EPSZ,HB19,HPW17,SGAP} {\it e.g.}, among which \cite{HPW17} obtained general results on local and global well-posedness of classical solutions to a version of \eqref{Single} in $\bbr^n$. 
\\
 \indent
 On the other hand, there are only a few studies of the adhesion model posed on the bounded domain \cite{BCE17,EPSZ, HB19}. Related to our result,  Hillen and Buttensch\"on  {were} the first to consider the well-posedness of the  initial-boundary problem of \eqref{Single} with the {explicit} sensing domain under the independent boundary condition in \cite{HB19}, where they constructed a bounded solution globally in time. Recently, Eckardt, Painter, Surulescu and Zhigun considered more general nonlocal models involving adhesion or nonlocal chemotaxis under the nonlocal Robin boundary condition in \cite{EPSZ} (For more complete review of the recent advancements of the nonlocal models, please see \cite{CPSZ}). It seems more realistic to consider the cell adhesion in a bounded domain, since cells interact with the boundaries of the regions in which they are contained.
  In this paper we extend the work from \cite{HB19} to a multidimensional bounded domain with two kinds of boundary conditions  satisfying a no-flux requirement.\\
\indent
 As mentioned in \cite{HB19}, there is another class of non-local models so called {\it aggregation equation}, where the non-local term has a singular interaction kernel in general. 
The aggregation equation arises as a gradient flow of a potential, and the well-posedness and blow-up features of the equations  are extensively studied  {\it{e.g.}} in  \cite{BLR, BL}. Recent  studies of such equation on a bounded domain can be found in \cite{FK,WS}. \\
\indent
Let  $u(x,t)$ and $v(x,t)$ denote the two types cancer cell population densities at spatial location $x\in \Omega$ and time $t$, where  $\Omega \subset\R^{n}$, $n\ge2$, is a bounded domain.  Following the simplifying assumptions in \cite{APS06}, we assume that the cell adhesion force is linearly proportional to the population density and  that  both cell types have the same sampling radius $R$; we set the velocities of nonlocal adhesions between two population as 
 \begin{align}\label{MODEL01}
 \begin{aligned}
\mathcal{K}[u,v](x,t)& =\int_{E(x)}\bkt{ M_{11}u(x+y,t)+  M_{12}v(x+y,t)      }\omega(y)dy,
\\
\mathcal{S}[u,v](x,t)& =\int_{E(x)}\bkt{ M_{21}u(x+y,t)+  M_{22}v(x+y,t)      }\omega(y)dy,
\\
 E(x)&: \mbox{ sensing domain depending on $x$, to be specified later, }
 \\
     \omega&: =(\omega_{1},\cdots, \omega_{n}) \mbox{ for } \, \omega_{i} \,\mbox{  bounded}.
\end{aligned} \end{align}
The positive constants $M_{11}$  and $M_{22}$  represent the self-adhesive strength of populations $u$  and $v$, respectively, and the positive constants $M_{12}$ and $M_{21}$ represent   the cross-adhesive strengths between the populations.  \\
\indent 
Below, we set the two initial-boundary problems on $u(x, t)$ and $v(x,t)$ according to boundary conditions.\\
{\bf I.} Nonlocal Robin boundary condition\\
Let  $\Omega$ be a $C^2$ smooth bounded domain and  $\calK[u, v]$, $\calS[u,v]$ be defined as in \eqref{MODEL01}. We consider the system
\begin{equation}\label{MODEL00}
 \left\{
\begin{array}{ll}
\vspace{1.5mm}
\partial_t u  - \Delta u= -\nabla \cdot (u \,\mathcal{K}[u,v]     )-mu +\displaystyle\frac{\lambda}{k} u(k-(u+v)), \qquad& x\in \Omega,\,\, t>0, \\
\vspace{1.5mm}
\partial_t v - \Delta v= -\nabla \cdot (v \,\mathcal{S}[u,v]     )+mu +\displaystyle\frac{\mu}{k} v(k-(u+v)), \qquad& x\in \Omega,\,\, t>0, \\
   \partial_{\nu} u =u\mathcal{K}[u,v]\cdot \nu,\quad \partial_{\nu} v=v\mathcal{S}[u,v]\cdot \nu, \qquad& x\in\partial \Omega,\,\, t>0, \\
u(x,0)=u_{0}(x),\,\, v(x,0)=v_{0}(x),\quad &x\in \Omega,
\end{array}
\right. 
\end{equation}
where $\nu(\cdot)$  is the $C^1$-smooth unit outward normal vector field to $\partial\Omega$. 
 The positive constants $m,k,\lambda$, and $\mu$ denote the mutation rate,  the carrying capacity, the growth rate of $u$, and the growth rate of $v$. The sensing domain $E(x)$ is given by 
  \begin{equation}\label{sensing1} E(x) = \bket{ y\in \R^{n} \,|\, x+y\in\Omega,\,\,|y|<R },  \end{equation}where $R$ is a fixed sensing radius. Moreover we assume 
 \begin{equation}\label{omega1}
\omega=(\omega_{1},\cdots, \omega_{n}) \mbox{  for }  \omega_{i} \mbox{  bounded and smooth. }
\end{equation}
\indent
  The equations \eqref{MODEL00} can be seen as a similar but simplified version of the model considered in \cite{BCE17}, 
describing dynamics of two type {of} cancer cells on a bounded domain in $\bbr^n$.  In \cite{BCE17} the terms $\mp m u$ refer that healthy or early stage cancer cells  mutate into later stage cells at a constant rate $m$. See \eqref{GENMODEL} in Remark \ref{Rmk1} for a full equations considered in \cite{BCE17}. 
\indent
On the other hand the similar reaction terms in \eqref{MODEL00} can be found in two-phenotypes cancer model undergoing
 {\it  epithelial -mesenchymal transition }(EMT). In this case we interpret $u$ and  $v$ as the density of epithelial tumor cells and the density of mesenchymal  tumor cells respectively.  The terms $\mp m u$ stand for EMT that epithelial phenotype cells undergo to acquire the ability to migrate into surroundings \cite{YY}. 
 Following this kind of view, the equation \eqref{MODEL00} should be understood to consider a local cancer cell invasion step according to terms in \cite{FLBC} and assumes the EMT takes place only. The logistic competition terms $u(k-(u+v))$ and $v(k-(u+v))$ describe the production of $u$ and $v$ are according to logistic law and  they compete for free space. Two-species cancer model with haptotactic invasion undergoing EMT is analyzed in \cite{DL20, GI}. For more details on the model we  refer to references therein.

Note that we require the total flux  zero on the boundary in \eqref{MODEL00}, which leads to a nonlocal boundary condition of Robin type,
\begin{equation}\label{totalzero} \partial_{\nu} u -u\mathcal{K}[u,v]\cdot \nu=0,\quad \partial_{\nu} v-v\mathcal{S}[u,v]\cdot \nu=0, \qquad x\in\partial \Omega,\,\, t>0. 
\end{equation}
  One way to attain the total flux zero condition is to assume for the solution of \eqref{MODEL00} to satisfy
\begin{align}\label{ind1}
&\pa_{\nu} u = 0, \quad \pa_{\nu} v = 0 \quad x\in \pa\Om, \\ \label{ind2}
 &\calK[u,v] =0, \quad \calS[u,v] =0 \quad x\in \pa\Om.
\end{align}
on the boundary, which is the boundary condition used in \cite{BCE17, HB19}. The first condition is the usual  Neumann  zero condition given on the solutions $u$ and $v$, however, \eqref{ind2} is the condition imposed to the nonlocal operators $\calK$ and $\calS$. 
Following \cite{HB19}, we call \eqref{ind1}, \eqref{ind2} {\it independent} or {\it zero-zero flux condition}.
An independent case allows the constant  solution 
\begin{equation}\label{constant}
u\equiv 0, \quad v\equiv k.
\end{equation}
  which is one of two non-negative constant solutions. The other one is $(0, 0)$.\footnote{ There is the other steady state, 
$u^{*}=-\frac{\mu}{\la}v^{*}, v^{*}=\frac{k(1-\frac{m}{\la})}{1-\frac{\mu}{\la}}$, which 
 are of different signs, hence unrealistic.} The $\mp m u$ terms in \eqref{MODEL00} suggest that  the later stage population dominates the total cell population if the mutation rate $m$ is large. The linear asymptotic stability around $(0, k)$ is considered in Appendix. The linear stability analysis for the related one dimensional model is performed in \cite{BE}.\\
 \indent
   We shall study the  {\it independent } case  with an explicit example of the sensing domain  $E(x)$ satisfying  $ |E(x)| = 0 $
 as $x$ approaches to $\pa \Om$ by which the condition \eqref{ind2} is assured. Note that the volume  of \eqref{sensing1} cannot vanish on the boundary.  We formulate the second initial-boundary problems under \eqref{ind1}--\eqref{ind2} as follows.
\\
{\bf II.} Zero-zero flux condition\\
 \indent Let $\Omega $ be the open ball of radius $L$, $B_L(0)$, and  $\calK[u, v]$, $\calS[u,v]$ be defined as in \eqref{MODEL01}. We have \begin{equation}\label{MODEL11}
 \left\{
\begin{array}{ll}
\vspace{1.5mm}
\partial_t u  - \Delta u= -\nabla \cdot (u \,\mathcal{K}[u,v]     )-mu +\displaystyle\frac{\lambda}{k} u(k-(u+v)), \qquad& x\in \Omega,\,\, t>0, \\
\vspace{1.5mm}
\partial_t v - \Delta v= -\nabla \cdot (v \,\mathcal{S}[u,v]     )+mu +\displaystyle\frac{\mu}{k} v(k-(u+v)), \qquad& x\in \Omega,\,\, t>0, \\
    \pa_{\nu} u = 0, \quad \pa_{\nu} v = 0, \quad
 \calK[u,v] =0, \quad  \calS[u,v] =0, \qquad& x\in \pa \Om\,\, t>0, 
  \\
u(x,0)=u_{0}(x),\,\, v(x,0)=v_{0}(x),\quad &x\in \Omega,
\end{array}
\right. 
\end{equation}
where $m, k, \mu,\la$ are same as before in  case {I}, and  $\nu(\cdot)$ denotes the unit outward normal vector to $\partial\Omega$.  The sensing domain $E(x)$ is given for $0<R<L$ by 
\begin{equation}\label{sensing2}  
E(x) = B_R(0)  \quad \mbox{ for } |x| < L-R, \quad E(x)= B_{L-|x|}(0) \mbox{ for }  L-R \le |x| < L,
\end{equation}
and we assume 
\begin{equation}\label{omega2}
 \omega(x):=\displaystyle\frac{x}{|x|}w(|x|) 
 \end{equation}
  for  $ w\in \mathcal{C}^{\infty}_{0}(\Omega)$  non-negative. 
 In \eqref{sensing2}  we choose  $E(x)$ as simple as possible  to satisfy the following property ; when $x$ is away from the boundary of the domain, $E(x)$ is $B_R(0)$  as was set for two dimensional model  in \cite{APS06}  .  When $x$ is close to the boundary, it shrinks to a smaller region such that $ x + {\mathbf r }\in  \Om$ for ${\mathbf r} \in E(x)$ and  $|E(x)| = 0$ as $x$ reaches to the boundary. 
  The choice of varying integration domain $E(x)$  affects the extent of regularity of the adhesion terms $\calK[u,v]$ and $\calS[u,v]$. In Lemma \ref{LEMKS} and Lemma \ref{LEMKS1}, we only show  $\calK[u,v]$ and $\calS[u,v]$ to be Lipschitz continuous however  smooth $u$, $v$, $\omega$ are.  
\footnote{ For case {II}, we could change the shrinking rate and shape of $E(x)$ as $x$ approaches the boundary  so that the regularity of adhesion terms is possibly worse.}\\
  \indent In case {I} the nonlocal nonlinear Robin type boundary condition as well as the restricted regularity of  $\calK[u,v]$ and $\calS[u,v]$ cause  difficulty in constructing local-in-time strong solutions directly by iterations. For the local well-posedness we take two steps;  we first construct the  generalized solution 
   relying on the semi-group theory for parabolic boundary value problem  found in  Amann's seminar works  \cite{A86,AM93}.
   In particular we  introduce a certain extension of the unit outward normal vector field, and employ the {\it generalized variation-of-constants formula} for this case. See Section~\ref{SEC21}, and  Section~\ref{SEC22}  for details. A similar construction can also be found in \cite{DGM-RS10,DM-RST10}. We can show the generalized solutions are indeed strong and satisfy the maximal regularity estimates employing the result of Denk-Hieber-Pr\"uss \cite{DHP07}.  The global well-posedness follows from  several  {\it a priori } estimates and Moser-Alikakos type estimate. 
\\
\indent
{In the zero-zero flux case the diffusive flux and the adhesion flux are independently zero on the boundary. 
The global well-posedness is obtained in a standard way though $\calK[u,v]$ and $\calS[u,v]$ are less regular than the adhesion terms in the nonlocal Robin type boundary case due to the shrinking sensing domain $E(x)$. As mentioned earlier, this case allows the constant steady state $(0,0)$ and $(0,k)$. We provide a linear stability analysis in the Appendix, where $(0,0)$ is found linearly unstable, and $(0,k)$ is linearly asymptotically stable if $m> \mu$.  Also we 
find that a Lyapunov type inequality holds when $v$ has no adhesion term and $\mu> \la$, see Remark $1$.}
\\
\indent
We are ready to state the main results in this paper. We first provide the global strong solvability   for case {I}.
\begin{theorem}\label{GETHM}
Let $\Omega \subset\R^{n}$, $n\ge2$, be an open bounded domain with $C^2$ boundary $\pa \Omega$. Suppose that the non-negative initial conditions $u_{0}$ and $v_{0}$ belong to $W^{2,p}(\Omega)$, $n<p <\infty$ and satisfy the compatibility condition
\[
\partial_{\nu}u_{0}-u_0\mathcal{K}[u_0,v_0]\cdot \nu=0,\qquad \partial_{\nu}v_{0}-v_0\mathcal{S}[u_0,v_0]\cdot \nu=0,\qquad\quad x\in\partial{\Omega}.
\]
Then,  \eqref{MODEL00}--\eqref{omega1}   admits a unique non-negative strong solution $(u,v)$ such that 
\[
u,v\in C([0,t);W^{1,p}(\Omega))\cap W^{1,p}(0,t; L^{p}(\Omega))\cap L^{p}(0,t; W^{2,p}(\Omega)), \qquad  t>0.
\]
Moreover, the solution $(u,v)$ has a boundedness property
\[
\sup_{t>0}( \| u(\cdot,t)   \|_{L^{\infty}(\Omega)}+ \| v(\cdot,t)   \|_{L^{\infty}(\Omega)}   )\le C.
\]
\end{theorem}

We next state the global strong solvability    for case {II}.
\begin{theorem}\label{GETHM1}
Let $\Omega\subset\R^{n}$, $n\ge2$ be an open ball of radius $L$, $B_L(0)$. Suppose that the non-negative initial data $u_{0}$ and $v_{0}$ belong to $W^{2,p}(\Omega)$, $n<p<\infty$, and satisfy the compatibility condition
\[
\partial_{\nu}u_{0}=\partial_{\nu}v_{0}=0,\qquad x\in\partial{\Omega}.
\]
Then,  \eqref{MODEL11}--\eqref{omega2} admits a unique non-negative strong solution $(u,v)$ such that  
\[
u,v\in C([0,t);W^{1,p}(\Omega))\cap W^{1,p}(0,t; L^{p}(\Omega))\cap L^{p}(0,t; W^{2,p}(\Omega)), \qquad  t>0.
\]
Moreover, the solution $(u,v)$ has a boundedness property
\[
\sup_{t>0}( \| u(\cdot,t)   \|_{L^{\infty}(\Omega)}+ \| v(\cdot,t)   \|_{L^{\infty}(\Omega)}   )\le C.
\]
\end{theorem}
\indent The paper is organized as follows.
We prove Theorem \ref{GETHM} in Section \ref{SEC2} and Theorem \ref{GETHM1} in Section \ref{SEC3}. Both sections start from proving preliminary results about the adhesion velocities
 $\calK[u,v]$, and $\calS[u,v]$.  The local well-posedness of \eqref{MODEL00}-\eqref{omega1} is proved in  Lemma \ref{LEM1}.
 The global well-posedness follows from the blow-up criteria (Lemma \ref{BUCLEM} and  Lemma \ref{LEMLINFUV}), where the solutions are found to be uniformly bounded in the $\| \cdot \|_{\iom}$ norm.  Section \ref{SEC3} is organized analogously; the local well-posedness of \eqref{MODEL11}-\eqref{omega2}
in Lemma \ref{LEM4} and the uniform boundedness in Lemma \ref{LEMLINFUV1}. Finally we provide a linear stability analysis in the Appendix for the steady state $(0,0)$, $(0, k)$ of the zero-zero case. 

Before closing this section we write down the related nonlocal $($local$)$ adhesion models arising  in cancer cell invasion \cite{BCE17, DL20, EPSZ} and explain simplifications and differences of our model compared to those studied in \cite{BCE17}.  
\begin{remark}\label{Rmk1}
A class of nonlocal(local) adhesion models proposed in \cite{BCE17}$($\cite{DL20}$)$ respectively 
can be written in an integrated manner  as
\begin{equation}\label{GENMODEL}
 \left\{
\begin{array}{ll}
\vspace{1.5mm}
\partial_t u_{1} = D_{1}\Delta u_{1} -\nabla \cdot (u_{1} \, F_{1}[u_{1},u_{2},f,c,b]     )-mu_{1} +G_{1}(u_{1}, u_{2},f,b), \\
\vspace{1.5mm}
\partial_t u_{2}   =D_{2}\Delta u_{2} -\nabla \cdot (u_{2} \,F_{2}[u_{1},u_{2},f,c,b]     )+mu_{1} +G_{2}(u_{1}, u_{2},f,b),\\
\vspace{1.5mm}
\partial_t f = -\alpha_{1}u_{1}f-\alpha_{2}u_{2}f \pm \alpha_{3}bf+\alpha_{4}f(\alpha_{5}-f),
\\
\vspace{1.5mm}
\partial_t c = \beta_{1}u_{1}+\beta_{2}u_{2}+ \beta_{3}bc- \beta_{4}c,
\\
\vspace{1.5mm}
\partial_t b = D_{3}\Delta b+\mu_{1}u_{1}+\mu_{2}u_{2}-\mu_{3}b.
\end{array}
\right. 
\end{equation}
Bitsouni  et al. \cite{BCE17} considered the system \eqref{GENMODEL} with nonlocal adhesion forces $F_{i}$, $i=1,2$. More precisely
\begin{equation}\label{remarkNA} F_i[u_1, u_2, f, c, b]:= \int_{\Om} K(|y-x|)g_i(u_1, u_2, f, c, b)(y) dy \end{equation}
for a kernel $K_i \in (W^{1, \infty}(\Om))^5$, where $g_i$ describes the nature of cell-cell and cell-matrix adhesive forces, chosen to be a linear function in $u_i$ and $f$ with bounded coefficients. The system describes the dynamics of two cancer cell populations coupled with ordinary differential equations describing the ECM (extracellular matrix) degradation, the production and decay of integrins $c$, and with an equation governing the evolution of the TGF-$\beta$ concentration. We denote 
by $u_1, u_2, f, c, b$ the density of healthy or early stage cancer cell population, the density of late stage cancer cell population, 
the ECM density, the level of integrins, and TGF-$\beta$ concentration. $\al_i, \be_i, \mu_i$ are given parameters. The first two equations also refer to homogeneous and heterogeneous cancer invasions. The movement of early stage cell is due to random diffusion and directed motility in response to adhesive forces, and the movement of late stage cell due to directed motility only $(D_2=0)$, exhibiting a heterogeneous type invasion. Note that $u_1$ mutates into $u_2$ at a constant rate $m$ in the model. On the boundary of the domain $\Om$, they assumed 
\begin{align}\label{bce17bd}
\nabla u_i\cdot \nu = F_i\cdot \nu = \nabla b \cdot \nu = 0, \quad i=1, 2.
\end{align}

 On the other hand,
Dai and Liu \cite{DL20} considered   the system \eqref{GENMODEL} with local adhesion forces $F_{i}=\chi_{i}\nabla f$, $i=1,2$  and $c\equiv0$,  where they modeled two species cancer invasion of surrounding healthy tissue. More precisely, their model involves the epithelial-mesenchymal transition $($EMT$)$ from epithelial phenotype cancer cells $(u_1)$  to mesenchimal phenotype cells $(u_2)$ at a constant rate $m$,  the migration of two cancer cell populations due to random diffusion $D_{i}>0$, $i=1,2$,  and haptotaxis $F_{i}$, $i=1,2$,  the proliferation of cancer cells $G_{i}$, $i=1,2$,   and their competition for space with ECM $(f)$, the production and decay of matrix metalloproteinases $(b)$, the degradation of ECM by matrix metalloproteinases upon contact, $-\alpha_{3}bf$, etc. On the boundary of the domain, they assumed $D_{i}\nabla u_{i}\cdot\nu-\chi_{i}u_{i}\nabla f\cdot\nu=\nabla b\cdot \nu=0$, $i=1,2$.

As is mentioned earlier,  Eckardt  et al. proposed the nonlocal models \footnote
{Eckardt et al. considered averaging nonlocal operators $\mathcal A_r$ and $\mathring{\nabla_r}$. Following notations in this paper, we have
\[\mathcal A_Ru(x) = \frac{1}{R}\frac{1}{|B_R(0)|} \int_{E(x)} u(x+y) \frac{y}{|y|} w(|y|) dy= \frac{1}{M_{11}R}\frac{1}{|B_R(0)|}\mathcal K [u,0] \]
 by choosing $\om(x)=\frac{y}{|y|} w(|y|)$ in \eqref{MODEL01}.
}
describing
   cell-cell and cell-tissue adhesions or nonlocal taxis (for a single cell phenotype) in \cite{EPSZ}. They proved global existence of  weak-strong solution and investigated solution properties for such complex systems, including non-linear motility coefficients. Moreover they verified a rigorous limit procedure linking nonlocal adhesion models and their local counterparts.
   \end{remark}
\begin{remark}\label{Rmk2}
 A global in time existence of a weak solution of the parabolic-hyperbolic system \eqref{GENMODEL} was shown in
 \cite{BCE17} under the zero-zero flux condition \eqref{bce17bd}. In this paper we simplify the system into \eqref{MODEL00} or 
\eqref{MODEL11} by ignoring the cell-matrix adhesion force as well as  effects of integrins and TGF-$\beta$ protein. Instead, we focus on mathematical analysis on nonlocal cell-cell adhesion under either the nonlocal Robin boundary condition, 
which was not addressed in \cite{BCE17} or \cite{HB19}, or the zero-zero flux condition \eqref{bce17bd}.  
The same analysis works as well without difficulty if we include the ECM equation and consider the nonlocal adhesion forces in the form of 
\eqref{remarkNA} with $c=b=0$. In the zero-zero conditions,  $F_i\cdot \nu=0$ 
is not usual since it is imposed on the nonlocal operator $F_i$, not on the solutions. In Section $3$ we work with a concrete
sensing domain to find that such boundary condition indeed makes sense. However, the choice of $E(x)$ seems made arbitrary and we have not yet found a more meaningful choice connected to the biological context considered here. 

Apart from the above simplifications, we assume $D_2$ strictly positive in the $u_2$-equation. Compared with \cite{BCE17}, we are interested in the boundary value problems with nonlocal Robin boundary condition or zero-zero flux condition and strong solutions for the boundary value problems. It is not clear to us whether the similar methods of our research can be applied for the case $D_2=0$.  When $D_2=0$,  Lagrangian coordinates $($\cite{DA, FU}$)$ seem to be more useful to construct a bounded solution along the characteristic curve if $u_1, u_2$ have same directed motility or velocity, $F_1, F_2$, which corresponds to
$M_{11}=M_{22}$, $M_{12}=M_{21}$ in \eqref{MODEL01}.
\end{remark}

\section{Nonlocal Robin boundary case}\label{SEC2}
\begin{subsection}{Preliminary}\label{SEC21}
We first consider a lemma on the extension $\mathcal{N}$ of the normal vector field $\nu$ into the domain $\Omega$. The   extension $\mathcal{N}$  is used in   Section~\ref{SEC22} to interpret  our nonlocal Robin boundary value problems as inhomogeneous Neumann  boundary value problems.
The lemma holds for a bounded domain with $C^k$-smooth boundary for any integer $k\ge 2$ without difficulty. Before stating the lemma, we remind the definition of a $C^k$-boundary in \cite[Appendix C.1]{EVANS}.
\begin{defn}\label{defbound}
Let $U \subset \R^n$ be an open bounded domain, and    $k\in \{1, 2, \dots\}$.  We call $\pa U$ is $C^k$ if
for each point $x^0 \in \pa U$ there exist $r >0$ and a $C^k$ function $\xi: \bbr^{n-1} \to \bbr$ 
such that upon an orthogonal change of coordinates we have
\[
U\cap B_r(x^0)= \{ x\in B_r(x^0) \, |\,  x_n > \xi(x_1, \dots, x_{n-1})\}.
\]
\end{defn}
\begin{lemma}\label{LEMNU}
Let $\Omega \subset\R^{n}$, $n\ge2$, be an open bounded domain with $C^2$ boundary $\pa \Omega$. Then, there exists at least one   vector field $ \mathcal{N}\in C^{1}(\overline{\Omega})$,  a continuous extension of the unit outward normal vector field $\nu$, such that
$\mathcal{N}=\nu$ on $\partial\Omega$. 
\end{lemma}
\begin{proof} 
 For $x^0 \in  \pa \Om$ there is $r, \xi$, and a relabeled tuple $(x_1, \dots, x_n)$ as in the definition such that $\pa U \cap B_r(x^0) = \{  x\in B_r(x^0) \,|\,  \xi(x_1, \dots, x_{n-1}) -x_n = 0\}$. 
The outward normal vector field $\nu(x)$ is well defined by
 \[ \nu(x) = \frac{ (\na \xi(\bar x), -1)}{ \sqrt{ 1 + |\na \xi(\bar x)|^2}}\]
 where $x= (\bar x, x_n)$.
\\
\indent Let us consider the function $E: \pa \Om \times \bbr \to \bbr^n$ given by 
\[ E(x, t) = x - t\nu(x). \]
Applying the tubular neighborhood theorem in Cahpter $10$ of \cite{Lee} there exist $\rho>0$  such that  $E$ is the $C^1$-diffeomorphism \footnote {We follow a definition of $C^1$-functions on a submanifold $\pa \Om \times \bbr$ embedded in $\bbr^n \times \bbr$.  The open set $E(V)$is called a tubular neighborhood of $\pa \Om$,
} on 
$ V: =\{(x, t) \in \pa \Om \times \bbr \,|\, |t| < \rho\}$ and 
 $T_{\rho}:=\{ E(x, t) | x\in \pa \Om, 0 < t <\rho \}$ is in $\Om$. 
We  find a smooth extension of $\nu$ on $\overline \Om$ as follows.  For $y\in T_{\rho}$
there is a unique $(x, t) \in \pa \Om \times \bbr$ such that $ y = x -t \nu(x)$, since $E$ is one to one. Note that the mapping $y \to (x, t) $ is $E^{-1}$, hence $C^1$.  We define a continuous extension of $\nu$ by 
\[ { N} (y) = 
\begin{cases} \frac{\rho -t}{\rho} \nu(x)\quad \mbox{ for } y\in T_{\rho} \\
 0 \quad \quad \quad \quad \mbox{ for } y \in \Om \setminus T_{\rho}.
 \end{cases} \]
  $N$ being $C^1$ on $T_{\rho}$ follows from the construction;  Define $\pi_1 : \pa \Om \times \bbr \to  \pa \Om $ to be the projection  onto the first slot, and $\pi_2$ onto the second slot.  Then we can write 
\[ N(y) =  \frac{\rho - \pi_2 \circ E^{-1}(y)}{\rho } \nu ( \pi_1\circ E^{-1}(y))\]
with $\nu(x) = \frac{ (\na \xi(\bar x), -1)}{ \sqrt{ 1 + |\na \xi(\bar x)|^2}}$ upon an orthogonal change of coordinates. 
Finally  smoothing out $N$ in $T_{\rho} \setminus T_{\frac{\rho}{2}}$ we obtain a smooth extension $\calN$ of $\nu$ on $\overline \Om$. 
\end{proof}

Before closing the section we prepare a $W^{1,\infty}(\Omega)$ estimate for the non-local  terms $\mathcal{K}$  and $\mathcal{S}$. 
\begin{lemma}\label{LEMKS}
Let $\Omega \subset\R^{n}$, $n\ge2$, be an open bounded domain with $C^2$ boundary $\pa \Omega$. Suppose that $f$, $g\in L^{1}(\Omega)$.
Assume further that   $\mathcal{K}[f,g]$, $\mathcal{S}[f,g]$, $E$, and $\omega$ are   given by  \eqref{MODEL01}, and  
 \eqref{sensing1}--\eqref{omega1}.
Then, there exists a  constant $C>0$ satisfying
\[\norm{\mathcal{K}[f,g]}_{W^{1,\infty}(\Omega)}+\norm{\mathcal{S}[f,g]}_{W^{1,\infty}(\Omega)}\le C(\norm{f}_{L^{1}(\Omega)}+ \norm{g}_{L^{1}(\Omega)}).
\]
\end{lemma}
\begin{proof}
It suffices to prove the Lipschitz continuity  of 
\[
\mathcal{I}[f](x)=\displaystyle\int_{E(x)} f(x+y)      \omega(y)dy,
\]
where $E(x)$ and $\omega$ are introduced in \eqref{sensing1}--\eqref{omega1}. By change of variable, $\mathcal{I}[f](x)$ can be written  as 
\[
\mathcal{I}[f](x)=\int_{V_{x}}f(z)      \omega(z-x)dz,\qquad V_{x}=\bket{ z\in\Omega \,|\,|z-x|<R   }.
\]
Choose any two points $x, y \in \Om$ and let $h= y-x$. We have
\[
\mathcal{I}[f](x+h)-\mathcal{I}[f](x)=\int_{V_{x+h}} f(z)\omega(z-x-h) dz-\int_{V_{x}} f(z)\omega(z-x) dz
\]
\[
=\int_{V_{x+h}} f(z)\bkt{\omega(z-x-h)-\omega(z-x)} dz+\int_{V_{x+h}} f(z)\omega(z-x) dz-\int_{V_{x}} f(z)\omega(z-x) dz.
\]
 We estimate
\[
\biggr{|}\int_{V_{x+h}} f(z)\bkt{\omega(z-x-h)-\omega(z-x)} dz\biggr{|}\le  C|h|\norm{f}_{L^{1}(\Omega)}\norm{\omega}_{W^{1,\infty}(\Omega)}.
\]
For the the other terms, we note that
 $ (V_{x+h}\cup V_{x}) \setminus (V_{x+h}\cap V_{x})$ is a subset of $B_R(x+h) \cup B_R(x) \setminus B_R(x+h) \cap B_R(x)$ in $\bbr^n$ and the volume of the latter is bounded by $C |h|$  with a uniform constant $C$ if  $y= x+h$ are in $B_R(x)$. We have
\[
\biggr{|}\int_{V_{x+h}} f(z)\omega(z-x) dz-\int_{V_{x}} f(z)\omega(z-x) dz\biggr{|}\le \norm{f}_{L^{1}(\Omega)}\norm{\omega}_{L^{\infty}(\Om) }\int_{(V_{x+h}\cup V_{x}) \setminus (V_{x+h}\cap V_{x})}1dz
\]
\[
\le C|h|\norm{f}_{L^{1}(\Omega)}\norm{\omega}_{L^{\infty} (\Omega)},
\]
hence
it  follows that
\[
\bigr{|}\mathcal{I}[f](x+h)-\mathcal{I}[f](x)\bigr{|}\le C|h|\norm{f}_{L^{1}(\Omega)}\norm{\omega}_{W^{1,\infty}(\Omega)}.
\]
This completes the proof.
\end{proof}

\end{subsection}
\begin{subsection}{Local well-posedness and blow-up criteria}\label{SEC22}
To prove a local well-posedness of \eqref{MODEL00}--\eqref{omega1}, we  shall employ {\it a generalized variation-of-constants formula} \eqref{INTF} for Neumann type parabolic boundary value problem  established by Amann (\cite{A86,AM93} {\it e.g.}).
Below, we introduce  the interpolation scale of spaces $E_{d}$ and the set of linear operators $A_{\delta}$ needed for writing down the formula. We work with $E_{d} $ and $A_{\delta}$  for $-1< d, \delta <0$, while they can be defined for $d, \delta>-1$. \\
\indent 
We denote  the boundary trace operator by $\gamma$. Note that
\begin{equation}\label{TRNOT}
\gamma\in\mathcal{L}( W^{1,p}(\Omega), W^{1-\frac{1}{p},p}(\partial\Omega)) \mbox{ for } 1< p<\infty.
\end{equation}
We also denote  the boundary trace of normal derivative by $B$;
\[ 
Bf:=  \ga \nabla f\cdot \nu  \, \mbox{ for }\,  f\in  W^{s, p}(\Om),\,1+\frac{1}{p}<s\le 2,\, 1<p<\infty.
\]
Consider the sectorial operator $A:= {(I-\Delta)}|_{D(A)}$     with its domain 
\[
D(A):=\bket{f\in W^{2,p}(\Omega) \, | \, Bf=0\,\mbox{ on }\,  \partial\Omega   }\, \mbox{ for }\, 1<p<\infty.
\]
Note that 
$ \inf Re\, \Sigma(A)>0$, where $\Sigma(A)$ is  the spectrum of $A$.  
In \cite[Section 4]{AM93} the pair $(A, B)$ is found to satisfy the condition of {\it normally elliptic problem}, which enables one to construct an interpolation scale of spaces. 
Let
\[
W^{s,p}_{B}:=\left\{
\begin{array}{lll}
\bket{f\in W^{s,p}(\Omega) \, | \, B f=0     }, \quad &\,\,\,\,1+\frac{1}{p}<s\le 2, \\
   W^{s,p}(\Omega),\quad &-1+\frac{1}{p}<s<1+\frac{1}{p}, \\
(W^{-s,p'}(\Omega))', \quad  &-2+\frac{1}{p}<s\le -1+\frac{1}{p},
\end{array}
\right. 
\]
and take $E_{0}=L^{p}(\Omega)=W^{0,p}_{B}$,  $E_{1}=W^{2,p}_{B}$.   Following  \cite[Section 6]{AM93},
  we construct an interpolation scale of spaces  
\[
E_{\theta}:=(E_{0},E_{1})_{\theta,p}=W^{2\theta,p}_{B}\]   for $2\theta\in(0,2)\setminus\{1,1+\frac{1}{p}\}$, where $(\cdot,\cdot)_{\theta,p}$ denotes the real interpolation functor.
 Introducing a completion of the normed space $(E_{0},\|A^{-1}\cdot\|_{E_{0}})$, which is denoted by $E_{-1}$, 
 we can inductively extend the definition of $E_{k+\theta}$ and $A_{k+\theta}$ for $-1< k+\theta<\infty$, $k=-1,0,1,\cdots$  (see \cite[(6.4)]{AM93}). Then, we have a family of operators 
\[
 A_{k+\theta}\in\mathcal{L}(E_{k+1+\theta},E_{k+\theta})
\]
  such that $-A_{k+\theta}$ is the infinitesimal generator of an analytic semigroup 
  \[
  \bket{e^{-tA_{k+\theta}} \,|\, t\ge0}\, \mbox{ on }\, E_{k+\theta},
  \]
and $A_{k+\theta}$ is a $W^{2(k+\theta), p}$- realization  
\footnote{ $A_{\beta}$ is a $W^{s,p}$- realization of $A$ if
$ A= A_{\beta}$ in $D(A)$ and  the range of  $A_{\beta}$ is in $W^{s,p}$. }of $A$ for {$ -1<k+\theta<\infty$, $k=-1,0,1,\cdots$, $0<\theta<1$, and $1<p<\infty$.} Let us  specify 
\begin{equation}\label{aal}
A_{\alpha-1}=W^{2(\alpha-1),p}_{B}\mbox{- realization of }\, A  \,\mbox{ for }\, 2\alpha\in \Bigr{(}1, 1+\frac{1}{p}\Bigr{)} \,\mbox{ with }\, n<p <\infty.
\end{equation}
The semigroup $e^{-t A_{\alpha-1}}$ satisfies the smoothing estimate (\cite[Lemma 3.1]{CM-R13}):\\
If $1<p<\infty$, $1<\beta<2\alpha<1+\frac{1}{p}$, $f\in W^{2\alpha-2,p}_{B}$, then there exist positive constants $\sigma=\sigma(\beta)<1$, $\kappa<1$ and $C(\alpha,\beta,p)$ such that
\begin{equation}\label{NONEST}
\norm{   e^{-tA_{\alpha-1}}f }_{W^{\beta,p}(\Omega)}\le Ct^{-\sigma}e^{-\kappa t}\norm{  f }_{W^{2\alpha-2,p}_{B}},\qquad t>0.
\end{equation}
 We will use \eqref{NONEST} in the proof of Lemma~\ref{LEM1} to  control the nonlinear terms.
Lastly, we define $B^c$ by  the continuous extension of $(B|_{Ker  (I-\Delta)} )^{-1}$ to $W^{2\alpha-1-\frac{1}{p}}(\partial\Omega)$.  Note that \begin{equation}\label{bc}
B^c\in\mathcal{L}(W^{2\alpha-1-\frac{1}{p}, \frac{1}{p}}(\partial\Omega), W^{2\al, p}(\Om))
\end{equation} for $\al, p$ of same range in \eqref{aal}.
\\
\indent
Let $\mathcal{N}\in C^{1}(\overline{\Omega})$ be a fixed  vector field satisfying $\mathcal{N}=\nu$ on $\partial\Omega$ constructed in Lemma~\ref{LEMNU}. Now 
 we consider the inhomogeneous Neumann  boundary value problems for \eqref{MODEL00}--\eqref{omega1}: 
\begin{equation}\label{NPBV0}
\begin{aligned}
&\partial_{t}u+(I-\Delta)u=g_{1},&  &\partial_{t}v+(I-\Delta)v=g_{2},   \quad\qquad&x\in\Omega, \,\,t\le T,\,\,\,&\\
&\partial_{\nu}u= h_1, &\qquad\quad  &\partial_{\nu} v=h_2,  &x\in\partial\Omega, \,\,t\le T,&\\
&u(x,0) =u_{0}(x),&\qquad  &v(x,0) =v_{0}(x),  &x\in\partial\Omega,\quad\qquad\,&
\end{aligned}
\end{equation}
and its   {\it generalized variation-of-constants formulas:} 
\begin{equation}\label{INTF}
  \begin{array}{ll}
u=e^{-tA_{\alpha-1}}u_{0}+\int_{0}^{t}e^{-(t-\tau)A_{\alpha-1}}(g_{1}(\tau)+ A_{\alpha-1} B^{c}  \gamma h_{1}(\tau))d\tau,\vspace{1mm}
\\
v=e^{-tA_{\alpha-1}}v_{0}+\int_{0}^{t}e^{-(t-\tau)A_{\alpha-1}}(g_{2}(\tau)+ A_{\alpha-1} B^{c} \gamma h_{2}(\tau))d\tau,
  \end{array}
\end{equation}
where
 \begin{equation}\label{INTF1}
  \begin{array}{ll}
g_{1}:=-\nabla \cdot (u \,\mathcal{K}[u,v]     )+(1-m)u +\displaystyle\frac{\lambda}{k} u(k-(u+v)),
\\
g_{2}:=-\nabla \cdot (v \,\mathcal{S}[u,v]     )+mu+v +\displaystyle\frac{\mu}{k} v(k-(u+v)),
\\
h_{1}:=u\mathcal{K}[u,v]\cdot\mathcal{\mathcal{N}},\qquad h_{2}:=v\mathcal{S}[u,v]\cdot\mathcal{\mathcal{N}}.
  \end{array}
\end{equation}
The formal argument to write  \eqref{NPBV0} into  \eqref{INTF} is presented in \cite[(11.16)--(11.20)]{AM93}.  
{\begin{defn}
Let $\Om, \calK, \calS$ and $\om$ be given as  for Theorem \ref{GETHM}. Assume $u_0$ and $v_0$ are functions belonging to $W^{2,p}(\Om), p>n$. We call $(u, v)$ in $(L^{\infty}(0, T; W^{1,p}(\Om)))^2$ satisfying the integral equation \eqref{INTF}, \eqref{INTF1} a \it{generalized solution} of 
\eqref{NPBV0} for $T>0$.
\end{defn}}
In what follows  we show that  \eqref{NPBV0}   has a unique local-in-time generalized 
 solution, which coincides with a strong solution satisfying maximal regularity estimates.

\begin{lemma}\label{LEM1}
Let $\Om, \calK, \calS$ and $\om$ be given as for Theorem \ref{GETHM}.
Assume that  $u_{0}$ and $v_{0}$  are  non-negative functions belonging to $W^{2,p}(\Omega)$, $p>n$, and satisfying
 the compatibility condition
\[
\partial_{\nu}u_{0}=u_0\mathcal{K}[u_0,v_0]\cdot \nu,\qquad \partial_{\nu}v_{0}=v_0\mathcal{S}[u_0,v_0]\cdot \nu,\qquad x\in\partial{\Omega}.
\]
 Then, there is a maximal time, $T_{\rm max}\le\infty$, such that a unique  non-negative strong  solution  $(u,v)$ of   \eqref{NPBV0}  exists and satisfies
\[
 u,v\in C([0,t]; W^{1,p}(\Omega))\cap W^{1,p}(0,t;L^{p}(\Omega))\cap L^{p}(0,t;W^{2,p}(\Omega)),\qquad t<T_{\rm max}. 
\]
Moreover, if $T_{\rm max}<\infty$, then
\begin{equation}\label{BUC}
\lim_{t\rightarrow T_{\rm max}}(\norm{u(\cdot,t)}_{W^{1,p}(\Omega)}+\norm{v(\cdot,t)}_{W^{1,p}(\Omega)})=\infty.
\end{equation}
\end{lemma}
\begin{proof}
We divide the proof into three parts. We first obtain a unique  generalized solution $(u,v)$, and then show that it is indeed a strong solution. In the last step, the non-negativity of solution components is shown.\\
 {\bf Step 1} (Generalized solution)
Let $n<p <\infty$. With  positive constants $T<1$ and $R_0$  to be specified below, we introduce the Banach space $X_{T}:=C([0,T];W^{1,p}(\Omega))$ and its closed convex subset $S_{T}\subset X_{T}$,
\[
S_{T}:=\bigr{\{}\,f\in X_{T}\,|\, \norm{f}_{L^{\infty}(0,T; W^{1,p}(\Omega))}\le  R_0     \bigr{\}}.
\]
Let $u,v \in S_{T}$,   $2\alpha\in (1, 1+\frac{1}{p})$  and $0<t<T$. As in \eqref{INTF}--\eqref{INTF1}, we consider \begin{equation}\label{PFPHI1}
\Phi_{1}(u,v):=e^{-tA_{\alpha-1}}u_{0}+\int_{0}^{t}e^{-(t-\tau)A_{\alpha-1}}(g_{1}(\tau)+ A_{\alpha-1} B^{c}  \gamma h_{1}(\tau))d\tau,
\end{equation}
\begin{equation}\label{PFPHI2}
\Phi_{2}(u,v):=e^{-tA_{\alpha-1}}v_{0}+\int_{0}^{t}e^{-(t-\tau)A_{\alpha-1}}(g_{2}(\tau)+ A_{\alpha-1} B^{c} \gamma h_{2}(\tau))d\tau,
\end{equation}
where  $\gamma$, $A_{\al-1}$ , and $B^{c}$  are as previously defined with  \eqref{TRNOT},  \eqref{aal} and  \eqref{bc}.  

We now show $\Phi_{1}(u,v),\Phi_{2}(u,v)\in S_{T}$.  Let $1<\beta<2\alpha$.
Using \eqref{NONEST} and $W^{\beta,p}(\Omega)\hookrightarrow W^{1,p}(\Omega)$, we compute
\[
\begin{aligned}
\|&\Phi_{1}(u,v)(t)\|_{W^{1,p}(\Omega)}\\
&\le \norm{e^{-tA_{\alpha-1}}u_{0}}_{W^{1,p}(\Omega)}+\int_{0}^{t}\norm{e^{-(t-\tau)A_{\alpha-1}}(g_{1}(\tau)+ A_{\alpha-1} B^{c}  \gamma h_{1}(\tau))}_{W^{1,p}(\Omega)}d\tau\\
&\le K_{1}\norm{u_{0}}_{W^{1,p}(\Omega)}+C\int_{0}^{t}\norm{e^{-(t-\tau)A_{\alpha-1}}(g_{1}(\tau)+ A_{\alpha-1} B^{c}  \gamma h_{1}(\tau))}_{W^{\beta,p}(\Omega)}d\tau\\
&\le  K_{1}\norm{u_{0}}_{W^{1,p}(\Omega)}+C\int_{0}^{t}e^{-\kappa(t-\tau)}(t-\tau)^{-\sigma}(\norm{g_{1}(\tau)}_{W^{2\alpha-2,p}_{B}}+ \norm{A_{\alpha-1} B^{c} \gamma h_{1}(\tau)}_{W^{2\alpha-2,p}_{B}})d\tau.
\end{aligned}
\]
Note that $A_{\alpha-1} B^{c}\gamma $ is well defined due to  Lemma~\ref{LEMKS}  and
\begin{equation}\label{BDRY}
 W^{1-\frac{1}{p},p}(\partial\Omega)\hookrightarrow W^{2\alpha-1-\frac{1}{p},p}(\partial\Omega),\qquad A_{\alpha-1}B^{c}\in\mathcal{L}(W^{2\alpha-1-\frac{1}{p},p}(\partial\Omega), W^{2\alpha-2,p}_{B}).
\end{equation}
\
Using $L^{p}(\Omega)=W^{0,p}_{B}\hookrightarrow W^{2\alpha-2,p}_{B}$,  Lemma~\ref{LEMKS}, and $W^{1,p}(\Omega)\hookrightarrow L^{\infty}(\Omega)$, we can estimate  
\[
\norm{g_{1}(\tau)}_{W^{2\alpha-2,p}_{B}}\le C\norm{g_{1}(\tau)}_{L^{p}(\Omega)}\le C(R_0+R_0^{2}).
\]
Using  Lemma~\ref{LEMKS} and   \eqref{BDRY}, we also compute   that
\[
 \norm{A_{\alpha-1} B^{c}  \gamma h_{1}(\tau))}_{W^{2\alpha-2,p}_{B}} \le C\norm{u\mathcal{K}[u,v](\tau)\cdot\mathcal{N}}_{W^{1,p}(\Omega)}\le CR_0^{2}.
\]
Then, combining the above computations leads to
\[
\norm{\Phi_{1}(u,v)(t)}_{W^{1,p}(\Omega)}\le K_{1}\norm{u_{0}}_{W^{1,p}(\Omega)}+C_{1}(R_0+R_0^{2})t^{1-\sigma},
\]
where $C_{1}$ is a positive constant independent of $R_0$.
Analogously to above,   we can  see that 
\[
\norm{\Phi_{2}(u,v)(t)}_{W^{1,p}(\Omega)}\le K_{2}\norm{v_{0}}_{W^{1,p}(\Omega)}+C_{2}(R_0+R_0^{2})t^{1-\sigma},
\]
where  $C_{2}$ is a positive constant independent of $R_0$. 
Choosing 
\[
R_0:=K_{1}\norm{u_{0}}_{W^{1,p}(\Omega)}+K_{2}\norm{v_{0}}_{W^{1,p}(\Omega)}+1,
\] 
\[
T<T_{1}:=\min\bket{ 1, \frac{1}{(2C_{1}(R_0+R_0^{2}))^{\frac{1}{1-\sigma}}}, \frac{1}{(2C_{2}(R_0+R_0^{2}))^{\frac{1}{ 1-\sigma}}} },
\]
and taking supremum over $0<t\le T$, we have
$
\Phi_{1}(u,v),\Phi_{2}(u,v)\in S_{T}$.

We next show the mapping $(u,v)\mapsto (\Phi_{1},\Phi_{2})$ is a contraction. Note from \eqref{PFPHI1}--\eqref{PFPHI2} that
\[
\Phi_{1}(u,v)(t)-\Phi_{1}(\tilde{u},\tilde{v})(t)=\int_{0}^{t}e^{-(t-\tau)A_{\alpha-1}}((g_{1}-\tilde{g}_{1})(\tau)+ A_{\alpha-1} B^{c}  \gamma (h_{1}-\tilde{h}_{1})(\tau))d\tau,
\]
\[
\Phi_{2}(u,v)(t)-\Phi_{2}(\tilde{u},\tilde{v})(t)=\int_{0}^{t}e^{-(t-\tau)A_{\alpha-1}}((g_{2}-\tilde{g}_{2})(\tau)+ A_{\alpha-1} B^{c}  \gamma (h_{2}-\tilde{h}_{2})(\tau))d\tau,
\]
where 
\[
\tilde{g}_{1}=-\nabla \cdot (\tilde{u} \,\mathcal{K}[\tilde{u},\tilde{v}]     )+(1-m)\tilde{u} +\displaystyle\frac{\lambda}{k} \tilde{u}(k-(\tilde{u}+\tilde{v})),\qquad \tilde{h}_{1}:=\tilde{u}\mathcal{K}[\tilde{u},\tilde{v}]\cdot\mathcal{\mathcal{N}},
\]
\[
\tilde{g}_{2}=-\nabla \cdot (\tilde{v} \,\mathcal{S}[\tilde{u},\tilde{v}]     )+ m \tilde{u}+\tilde{v} +\displaystyle\frac{\mu}{k} \tilde{v}(k-(\tilde{u}+\tilde{v})),\qquad \tilde{h}_{2}:=\tilde{v}\mathcal{S}[\tilde{u},\tilde{v}]\cdot\mathcal{\mathcal{N}}.
\]
Then, by similar  computations as above, we have
\[
\sup_{t\le T}\norm{(\Phi_{1},\Phi_{2})(u,v)(t)-(\Phi_{1},\Phi_{2})(\tilde{u},\tilde{v})(t)}_{W^{1,p}(\Omega)}
\]
\[
\le C_{3}(R_0+1)T^{1-\sigma}\sup_{t\le T}\norm{ (u,v)(t)- (\tilde{u},\tilde{v})(t)}_{W^{1,p}(\Omega)}, 
\]
where $C_{3}>0$ is a constant independent of $R_0$.
Taking 
\[
T<T_{2}:=\min \bket{ T_{1},  \frac{1}{(2C_{3}(R_0+1))^{\frac{1}{1-\sigma}}} },
\]
we obtain 
\[
\sup_{t\le T}\norm{(\Phi_{1},\Phi_{2})(u,v)(t)-(\Phi_{1},\Phi_{2})(\tilde{u},\tilde{v})(t)}_{W^{1,p}(\Omega)}\le \frac{1}{2} \sup_{t\le T}\norm{ (u,v)(t)- (\tilde{u},\tilde{v})(t)}_{W^{1,p}(\Omega)},
\]
i.e., the mapping is a contraction. According to the Banach fixed point theorem, this mapping has a fixed point in $S_{T}$, denoted again $(u,v)$.  Thus, the generalized solution $(u,v)$ for \eqref{NPBV0} is obtained. 
By the standard extension argument and the fact that the above choice of $T$ depends only on $\|u_0\|_{W^{1,p}(\Omega)}$, and $\|v_0\|_{W^{1,p}(\Omega)}$, it should be noted that there exists $T_{\rm M}\le \infty$ such that $ u,v\in C([0,T_{M}); W^{1,p}(\Omega))$, and 
\begin{equation}\label{TMP}
\mbox{either }\, T_{\rm M}=\infty,\,\mbox{ or }\, \lim_{t\rightarrow T_{\rm M}}(\norm{u(\cdot,t)}_{W^{1,p}(\Omega)}+\norm{v(\cdot,t)}_{W^{1,p}(\Omega)})=\infty.
\end{equation}

We next show the uniqueness. Let $T<T_{M}$, and let $(u,v)$ and $(\tilde{u}, \tilde{v})$ be two constructed solutions for $t\le T$.
Analogously to above, we can estimate
\[
\begin{aligned}
\|( u&-\tilde{u})(t)\|_{W^{1,p}(\Omega)}\\
&\le \int_{0}^{t} e^{-\kappa(t-\tau)}(t-\tau)^{-\sigma}(\norm{( g_{1}-\tilde{g}_{1})(\tau)}_{L^{p}(\Omega)}+\|A_{\alpha-1}B^{c}\gamma( h_{1}-\tilde{h}_{1})(\tau)\|_{L^{p}(\Omega)} ) d\tau,
\end{aligned}
\]
\[
\begin{aligned}
\|( g_{1}&-\tilde{g}_{1})(\tau)\|_{L^{p}(\Omega) }
\\
&\le C\sup_{t\le T}(\norm{u(t)}_{W^{1,p}(\Omega)}+ \norm{v(t)}_{W^{1,p}(\Omega)}+ \norm{\tilde{u}(t)}_{W^{1,p}(\Omega)}+\norm{\tilde{v}(t)}_{W^{1,p}(\Omega)})
\\
&\qquad\times( \|( u-\tilde{u})(\tau)\|_{W^{1,p}(\Omega)}+ \|( v-\tilde{v})(\tau)\|_{W^{1,p}(\Omega)})
\\
&\le C( \|( u-\tilde{u})(\tau)\|_{W^{1,p}(\Omega)}+ \|( v-\tilde{v})(\tau)\|_{W^{1,p}(\Omega)}),
\end{aligned}
\]
and
\[
\begin{aligned}
\| A_{\alpha-1}B^{c}\gamma( h_{1}&-\tilde{h}_{1})(\tau)\|_{L^{p}(\Omega)}\\
&
\le \|(h_1-\tilde{h}_1)(\tau)\|_{W^{1,p}}(\Omega)\\
&\le C\sup_{t\le T}(\norm{u(t)}_{W^{1,p}(\Omega)}+ \norm{v(t)}_{W^{1,p}(\Omega)}+ \norm{\tilde{u}(t)}_{W^{1,p}(\Omega)}+\norm{\tilde{v}(t)}_{W^{1,p}(\Omega)})
\\
&\qquad\times( \|( u-\tilde{u})(\tau)\|_{W^{1,p}(\Omega)}+ \|( v-\tilde{v})(\tau)\|_{W^{1,p}(\Omega)})
\\
&\le C( \|( u-\tilde{u})(\tau)\|_{W^{1,p}(\Omega)}+ \|( v-\tilde{v})(\tau)\|_{W^{1,p}(\Omega)}).
\end{aligned}
\]
Thus, we have
\[
\|( u-\tilde{u})(t)\|_{W^{1,p}(\Omega)}\le C \int_{0}^{t} e^{-\kappa(t-\tau)}(t-\tau)^{-\sigma}( \|( u-\tilde{u})(\tau)\|_{W^{1,p}(\Omega)}+ \|( v-\tilde{v})(\tau)\|_{W^{1,p}(\Omega)})d\tau.
\]
Similarly, we also have 
\[
\|( v-\tilde{v})(t)\|_{W^{1,p}(\Omega)}\le C \int_{0}^{t} e^{-\kappa(t-\tau)}(t-\tau)^{-\sigma}( \|( u-\tilde{u})(\tau)\|_{W^{1,p}(\Omega)}+ \|( v-\tilde{v})(\tau)\|_{W^{1,p}(\Omega)})d\tau.
\]
Adding the above two estimates and 
using a Gr\"onwall type inequality,   $(u,v)=(\tilde{u},\tilde{v})$ is obtained for $t\le T$. Since $T<T_{\rm M}$ is arbitrary, we have the uniqueness of solutions.\\
 {\bf Step 2} (Strong solution)
 We next consider the regularity of the constructed solution $(u,v)$. Let $t\le T<T_{M}$.
 As $1<\beta<2\alpha<1+\frac{1}{p}$, we first note that
\begin{equation}\label{INITWBP}
\| e^{-tA_{\alpha-1}}u_{0}\|_{W^{\beta,p}(\Omega)}\le C\|u_{0}\|_{W^{2,p}(\Omega)}, \qquad \| e^{-tA_{\alpha-1}}v_{0}\|_{W^{\beta,p}(\Omega)}\le C\|v_{0}\|_{W^{2,p}(\Omega)}.
\end{equation}
If we replace the computations for the initial counterparts  in the previous step  by \eqref{INITWBP}, then we    have   
$
 u,v\in C ([0,T]; W^{\beta,p  }_{B})$.
Thus, as in  \cite[(3.5)]{A86}, it can be shown that 
\[
  u,v\in C^{1}([0,T]; W^{\beta-2,p  }_{B})\quad\mbox{ as well as}\quad u,v\in C^{\frac{1}{2}-\frac{1}{2p}}([0,T]; W^{\beta-1+\frac{1}{p}, p  }_{B}).
  \] Therefore, we  have $u,v\in W^{\frac{1}{2}-\frac{1}{2p},p}(0,T; W^{\beta-1, p  }(\partial\Omega))$ and, in particular, 
  \[
h_{1},h_{2}\in W^{\frac{1}{2}-\frac{1}{2p},p}(0,T; W^{\beta-1, p  }(\partial\Omega)).
  \]
Due to these facts along with  $h_{1}, h_{2}\in  L^{\infty}(0,T; W^{1-\frac{1}{p},p}(\partial\Omega))$ and  $g_{1}, g_{2}\in  L^{\infty}(0,T; L^{p}(\Omega))$, by
applying the maximal regularity theorem \cite[Theorem 2.1]{DHP07}  to  
\begin{equation}\label{NPBV00}
\begin{aligned}
&\partial_{t}f_{1}+(I-\Delta)f_{1}=g_{1},&  &\partial_{t}f_{2}+(I-\Delta)f_{2}=g_{2},   \quad\qquad&x\in\Omega, \,\,t\le T,\,\,\,&\\
&\partial_{\nu}f_{1}= h_1, &\qquad\quad  &\partial_{\nu} f_{2}=h_2,  &x\in\partial\Omega, \,\,t\le T,&\\
&f_{1}(x,0) =u_{0}(x),&\qquad  &f_{2}(x,0) =v_{0}(x),  &x\in\partial\Omega,\quad\qquad\,&
\end{aligned}
\end{equation}
we have the unique strong solution $(f_1,f_2)$ for \eqref{NPBV00} in the class
\[
f_1,f_2\in W^{1,p}(0,T;L^{p}(\Omega))\cap L^{p}(0,T;W^{2,p}(\Omega)),\qquad T<T_{\rm M}.
\]
Since the strong solution is the generalized solution (\cite{A86}, \cite{LaTri}), we put  \eqref{NPBV00} into the generalized variation-of-constants formulas with respect to $f_1$, $f_2$ with the same right hand side terms as in \eqref{INTF}, \eqref{INTF1}. 
Then we have   $(f_1,f_2)=(u,v)$  due to the  uniqueness of the generalized solution  and thus,
\[
u,v\in W^{1,p}(0,T;L^{p}(\Omega))\cap L^{p}(0,T;W^{2,p}(\Omega)),\qquad T<T_{\rm M}.
\]
 {\bf Step 3} (Non-negativity)
It remains to show the non-negativity of the constructed solution $(u,v)$. Let $t\le T<T_{M}$. Define $u_{-}:=-\min\bket{u,0}$.
Multiplying  the $u$-equation in \eqref{NPBV0} by $u_{-}$ and integrating over $\Omega$,  using the direct computation and Young's inequality, we can compute 
\begin{equation}\label{UNU}
\begin{aligned}
\frac{1}{2}\frac{d}{dt}&\int_{\Omega}|u_{-}|^{2}+\int_{\Omega}|\nabla u_{-}|^{2}\\
&
\le \frac{1}{2}\int_{\Omega}|\nabla u_{-}|^{2}+C(1+ \|\mathcal{K}[u,v]\|_{L^{\infty}(\Omega)}^{2}+\|u\|_{L^{\infty}(\Omega)}+\|v\|_{L^{\infty}(\Omega)})\int_{\Omega}|u_{-}|^{2},\qquad t\le T,
\end{aligned}
\end{equation}
where we used Lemma~\ref{LEMKS} and $W^{1,p}(\Omega)\hookrightarrow L^{\infty}(\Omega)$.
Then,  $u\ge0$  can be obtained by using Gr\"onwall's lemma.  Similarly, testing the $v$-equation in  \eqref{NPBV0} by $v_-:=- \min\{v,0\}$ and using Young's inequality, we observe that 
\begin{equation}\label{UNV}
\begin{aligned}
\frac{1}{2}\frac{d}{dt}&\int_{\Omega}|v_{-}|^{2}+\int_{\Omega}|\nabla v_{-}|^{2} \\
&\le \frac{1}{2}\int_{\Omega}|\nabla v_{-}|^{2}+C(1+ \|\mathcal{S}[u,v]\|_{L^{\infty}(\Omega)}^{2}+\|u\|_{L^{\infty}(\Omega)}+\|v\|_{L^{\infty}(\Omega)})\int_{\Omega}|v_{-}|^{2}-m\int_{\Omega}uv_{-}
\\
&\le \frac{1}{2}\int_{\Omega}|\nabla v_{-}|^{2}+C\int_{\Omega}|v_{-}|^{2},\quad t\le T.
\end{aligned}
\end{equation}
Again by Gr\"onwall's lemma, we have $v\ge0$.  Since $T<T_{\rm M}$ is arbitrary, we obtain the non-negativity of solutions.
Finally taking $T_{\rm max}$ as $T_{\rm M}$  in \eqref{TMP} completes the proof.
\end{proof}

We next prove a refined blow-up criteria.
\begin{lemma}\label{BUCLEM}
Let the same assumptions as in Lemma~\ref{LEM1} be satisfied. The solution  $(u,v)$ of  \eqref{NPBV0} given by Lemma~\ref{LEM1} satisfies
\[
\mbox{either }\, T_{\rm max}=\infty,\,\mbox{ or }\, \lim_{t\rightarrow T_{\rm max}}(\norm{u(\cdot,t)}_{L^{\infty}(\Omega)}+\norm{v(\cdot,t)}_{L^{\infty}(\Omega)})=\infty.
\]
\end{lemma}
\begin{proof}
By \eqref{BUC}, it suffices to show that if  
\begin{equation}\label{BUCPF0}
\norm{u}_{L^{\infty}(0,T_{\rm max};L^{\infty}(\Omega))}+\norm{v}_{L^{\infty}(0,T_{\rm max};L^{\infty}(\Omega))}\le C,\qquad T_{\rm max}<\infty,
\end{equation}
 then
\begin{equation}\label{BUCPF1}
\norm{u}_{L^{\infty}(0,T_{\rm max};W^{1,p}(\Omega))}+\norm{v}_{L^{\infty}(0,T_{\rm max};W^{1,p}(\Omega))}\le C.
\end{equation}
Let $p>n$, $2\alpha\in (1, 1+\frac{1}{p})$, $1<\beta<2\alpha$, and let $\varepsilon>0$ be a number such that $0<T_{\rm max}-\varepsilon<t<T_{\rm max}$. 
Using \eqref{INTF}, \eqref{NONEST} and $W^{\beta,p}(\Omega)\hookrightarrow W^{1,p}(\Omega)$, we compute
\[
\begin{aligned}
\|u&(t)\|_{W^{1,p}(\Omega)} \\
&\le K_{1}\norm{u(T_{\rm max}-\varepsilon)}_{W^{1,p}(\Omega)}\\
&\quad+C\int_{T_{\rm max}-\varepsilon}^{t}\|e^{-(t-\tau)A_{\alpha-1}}(g_{1}(\tau)+ A_{\alpha-1} B^{c} \gamma h_{1}(\tau))\|_{W^{\beta,p}(\Omega)}d\tau\\
&\le  K_{1}\norm{u(T_{\rm max}-\varepsilon)}_{W^{1,p}(\Omega)}\\
&\quad+C\int_{T_{\rm max}-\varepsilon}^{t}e^{-\kappa(t-\tau)}(t-\tau)^{-\sigma}(\norm{g_{1}(\tau)}_{W^{2\alpha-2,p}_{B}}+ \norm{A_{\alpha-1} B^{c} \gamma h_{1}(\tau))}_{W^{2\alpha-2,p}_{B}})d\tau,
\end{aligned}
\]
where   $\gamma\in\mathcal{L}( W^{1,p}(\Omega), W^{1-\frac{1}{p},p}(\partial\Omega))$ is the boundary trace operator, $B^{c}$ is the continuous extension of $(B|_{Ker(I-\Delta)})^{-1}$ to $W^{2\alpha-1-\frac{1}{p}}(\partial\Omega)$, and $h_{1}$, $g_{1}$ are given in \eqref{INTF1}.
   Using $L^{p}(\Omega)=W^{0,p}_{B}\hookrightarrow W^{2\alpha-2,p}_{B}$,  Lemma~\ref{LEMKS}, \eqref{BDRY}, and \eqref{BUCPF0}, we compute
\[
\begin{aligned}
\|g_{1}&(\tau)\|_{W^{2\alpha-2,p}_{B}}\\
&\le C\norm{g_{1}(\tau)}_{ L^{p}(\Om)}\\
 &\le C(\norm{u(\tau)}_{L^{\infty}(\Omega)}+\norm{u(\tau)}_{L^{\infty}(\Omega)}^{2}+\norm{u(\tau)}_{L^{\infty}(\Omega)}\norm{v(\tau)}_{L^{\infty}(\Omega)})\\
&\quad+C(\norm{u(\tau)}_{W^{1,p}(\Omega)}\norm{\mathcal{K}[u,v](\tau)}_{L^{\infty}(\Omega)}+\norm{u(\tau)}_{L^{\infty}(\Omega)}\norm{\mathcal{K}[u,v](\tau)}_{W^{1,p}(\Omega)})\\
 &\le C
+C(\norm{u(\tau)}_{W^{1,p}(\Omega)}+\norm{v(\tau)}_{W^{1,p}(\Omega)}),
\end{aligned}
\]
and 
\[
\begin{aligned}
\|A_{\alpha-1}& B^{c}  \gamma h_{1}(\tau))\|_{W^{2\alpha-2,p}_{B}}\\
&
 \le C\norm{u\mathcal{K}[u,v](\tau)\cdot\mathcal{N}}_{W^{1,p}(\Omega)}\\
 &
 \le C(\norm{u(\tau)}_{W^{1,p}(\Omega)}\norm{\mathcal{K}[u,v](\tau)}_{L^{\infty}(\Omega)}+\norm{u(\tau)}_{L^{\infty}(\Omega)}\norm{\mathcal{K}[u,v](\tau)}_{W^{1,p}(\Omega)})\\
  &\le C
+C(\norm{u(\tau)}_{W^{1,p}(\Omega)}+\norm{v(\tau)}_{W^{1,p}(\Omega)}).
\end{aligned}
\]
It then follows that
\[
\begin{aligned}
\|u&(t)\|_{W^{1,p}(\Omega)} \\
&\le K_{1}\norm{u(T_{\rm max}-\varepsilon)}_{W^{1,p}(\Omega)}
+C_{4}\varepsilon^{1-\sigma}\sup_{T_{\rm max}-\varepsilon<t<T_{\rm max}}(\norm{u(t)}_{W^{1,p}(\Omega)}+\norm{v(t)}_{W^{1,p}(\Omega)}),
\end{aligned}
\]
where $C_{4}$ is a positive constant independent of $\varepsilon$.
 Analogously to above, we can compute  
\[
\begin{aligned}
\|v&(t)\|_{W^{1,p}(\Omega)} \\
&\le K_{2}\norm{v(T_{\rm max}-\varepsilon)}_{W^{1,p}(\Omega)}
+C_{5}\varepsilon^{1-\sigma}\sup_{T_{\rm max}-\varepsilon<t<T_{\rm max}}(\norm{u(t)}_{W^{1,p}(\Omega)}+  \norm{v(t)}_{W^{1,p}(\Omega)}),
\end{aligned}
\]
where $C_{5}$ is a positive constant independent of $\varepsilon$. Adding above two inequalities and taking supremum over $T_{\rm max}-\varepsilon<t<T_{\rm max}$, we have 
\[
\begin{aligned}
\sup_{T_{\rm max}-\varepsilon<t<T_{\rm max}}\bigr{(}\| &u(t)\|_{W^{1,p}(\Omega)} +\|v(t)\|_{W^{1,p}(\Omega)} \bigr{)}\\
&\le K_{1}\norm{u(T_{\rm max}-\varepsilon)}_{W^{1,p}(\Omega)}+K_{2}\norm{v(T_{\rm max}-\varepsilon)}_{W^{1,p}(\Omega)}\\
&\quad
+(C_{4}+C_{5})\varepsilon^{1-\sigma}\sup_{T_{\rm max}-\varepsilon<t<T_{\rm max}}\bigr{(}\norm{u(t)}_{W^{1,p}(\Omega)} +\norm{v(t)}_{W^{1,p}(\Omega)} \bigr{)}.
\end{aligned}
\]
Therefore,  taking  sufficiently small $\varepsilon$, \eqref{BUCPF1} is obtained. This completes the proof.
\end{proof}

\end{subsection}

\begin{subsection}{ A priori estimates}\label{SEC23}
Next, we provide some {\it a priori} estimates (Lemma~\ref{LEML1} and  Lemma~\ref{LEMLINFUV}).

\begin{lemma}\label{LEML1}
Let the same assumptions as in Lemma~\ref{LEM1} be satisfied.  The solution  $(u,v)$ of \eqref{NPBV0} given by Lemma~\ref{LEM1} for $T<T_{\rm max}$ satisfies
\begin{equation}\label{UL1}
\sup_{t\le T}\int_{\Omega}u(\cdot,t)\le C(\norm{u_{0}}_{L^{1}(\Omega)}),
\end{equation}
\begin{equation}\label{VL1}
\sup_{t\le T}\int_{\Omega}v(\cdot,t)\le C(\norm{u_{0}}_{L^{1}(\Omega)}, \norm{v_{0}}_{L^{1}(\Omega)}),
\end{equation}
\begin{equation}\label{KLINF}
\sup_{t\le T} \norm{\mathcal{K}[u,v](\cdot, t)}_{L^{\infty}(\Omega)}\le C(\norm{u_{0}}_{L^{1}(\Omega)},  \norm{v_{0}}_{L^{1}(\Omega)}),
\end{equation}
and
\begin{equation}\label{SLINF}
\sup_{t\le T} \norm{\mathcal{S}[u,v](\cdot, t)}_{L^{\infty}(\Omega)}\le C(\norm{u_{0}}_{L^{1}(\Omega)},  \norm{v_{0}}_{L^{1}(\Omega)}).
\end{equation}
\end{lemma}
\begin{proof}
Integrating $\eqref{MODEL00}_{1}$ and $\eqref{MODEL00}_{2}$ over $\Omega$, we obtain
\begin{equation}\label{LEML1_1}
\frac{d}{dt}\int_{\Omega}u+\frac{\lambda}{k}\int_{\Omega}u(u+v)=(\lambda-m)\int_{\Omega}u,
\end{equation}
and
\begin{equation}\label{LEML1_2}
\frac{d}{dt}\int_{\Omega}v+\frac{\mu}{k}\int_{\Omega}v(u+v)=m\int_{\Omega}u+ \mu\int_{\Omega}v.
\end{equation}
As we have 
\[
(|\lambda-m|+1)\int_{\Omega}u\le \frac{\lambda}{k}\int_{\Omega}u^{2}+C,
\]
it follows from \eqref{LEML1_1}    that
\[
y'+ y\le C,\qquad y(t):=\int_{\Omega}u(\cdot,t).
\]
Therefore,  \eqref{UL1} is obtained  by standard ODE argument.
Similarly, as Young's inequality gives
\[
(\mu+1)\int_{\Omega}v\le \frac{\mu}{k}\int_{\Omega}v^{2}+C,
\]
it follows from  \eqref{UL1}  and \eqref{LEML1_2} that
\[
\frac{d}{dt}\int_{\Omega}v+\int_{\Omega}v\le C(\norm{u_{0}}_{L^{1}(\Omega)}).
\]
Thus, we can also obtain \eqref{VL1}. Then, \eqref{KLINF} and \eqref{SLINF} are direct consequences of \eqref{UL1}--\eqref{VL1} and \eqref{omega1}. 
This completes the proof.
\end{proof}
\begin{lemma}\label{LEMLINFUV}
Let the same assumptions as in Lemma~\ref{LEM1} be satisfied.  The solution  $(u,v)$ of  \eqref{NPBV0}  given by Lemma~\ref{LEM1} for $T<T_{\rm max}$ satisfies
\begin{equation}\label{ULINFTY}
\sup_{t\le T}\norm{u(\cdot,t)}_{L^{\infty}(\Omega)}\le  C(\norm{u_{0}}_{(L^{1}\cap L^{\infty})(\Omega)}, \norm{v_{0}}_{L^{1}(\Omega)}),
\end{equation}
and
\begin{equation}\label{VLINFTY}
\sup_{t\le T}\norm{v(\cdot,t)}_{L^{\infty}(\Omega)}\le  C(\norm{u_{0}}_{(L^{1}\cap L^{\infty})(\Omega)}, \norm{v_{0}}_{(L^{1}\cap L^{\infty})(\Omega)}).
\end{equation}
\end{lemma}
\begin{proof}
Let $p>1$ and $t\le T<T_{\rm max}$.
Multiplying $\eqref{MODEL00}_{1}$ by $u^{p-1}$, integrating over $\Omega$, and using integrating by parts, we have
\[
\frac{1}{p}\frac{d}{dt}\int_{\Omega}u^{p}+\frac{4(p-1)}{p^{2}}\int_{\Omega}|\nabla u^{\frac{p}{2}}|^{2}+\frac{\lambda}{k}\int_{\Omega} u^{p}(u+v)
=(\lambda-m)\int_{\Omega}u^{p}+(p-1)\int_{\Omega}u^{p-1}\nabla u \,\mathcal{K}[u,v].
\]
Using Young's inequality and \eqref{KLINF}, we can compute the rightmost term as
\[
(p-1)\int_{\Omega}u^{p-1}\nabla u \,\mathcal{K}[u,v]\le \frac{p-1}{p^{2}}\int_{\Omega}|\nabla u^{\frac{p}{2}}|^{2}+C(p-1)\int_{\Omega}u^{p},
\]
and thus, it follows that
\[
\frac{1}{p}\frac{d}{dt}\int_{\Omega}u^{p}+\frac{3(p-1)}{p^{2}}\int_{\Omega}|\nabla u^{\frac{p}{2}}|^{2} 
\le  C_{6} (p+1)\int_{\Omega}u^{p},
\]
where  $C_{6}>0$  is a constant independent of $p$. Then, \eqref{ULINFTY} can be deduced by Moser-Alikakos iteration argument \cite{A79}.  Indeed,  for $p=p_{k}:=2^{k}$, $k=1,2,3,\cdots$, the last inequality becomes
\begin{equation}\label{MA_1}
 \frac{d}{dt}\int_{\Omega}u^{p_k}+\frac{3(p_k-1)}{p_k}\int_{\Omega}|\nabla u^{ p_{k-1}} |^2\le C_{6}p_k(p_k+1)\int_{\Omega}u^{p_k}.
\end{equation}
Using the Gagliardo-Nirenberg interpolation inequality and Young's inequality, we note that
\[
\norm{u^{ p_{k-1}}}_{L^{2}(\Omega)}^{2}\le C\norm{ u^{ p_{k-1}} }_{L^{1}(\Omega)}^{\frac{4}{n+2}}\norm{\nabla u^{ p_{k-1}} }_{L^{2}(\Omega)}^{\frac{2n}{n+2}}+C\norm{ u^{ p_{k-1}} }_{L^{1}(\Omega)}^{2}
\]
\[
\le \frac{1}{2C_{6}p_{k}(p_k+1)}\norm{\nabla u^{ p_{k-1}} }_{L^{2}(\Omega)}^{2}+Cp_{k}^{\frac{n(n+2)}{n+1}}\norm{ u^{ p_{k-1}} }_{L^{1}(\Omega)}^{2}
\]
for some constant $C>0$   independent of $p_{k}$. Plugging it into \eqref{MA_1}, due to   $\frac{3(p_k-1)}{p_k}\ge \frac{3}{2}$, we have
\begin{equation}\label{MA_2}
 \frac{d}{dt}\int_{\Omega}u^{p_k}+2C_{6}p_k(p_k+1)\int_{\Omega}u^{p_k}\le C_{7}p_{k}^{\frac{n(n+2)}{n+1}+2}\bke{  \int_{\Omega} u^{ p_{k-1}}  }^{2},
\end{equation}
where   $C_{7}>0$ is a constant independent of $p_k$. We take  a sufficiently large constant $C_{8}\ge1$  independent of $p_{k}$ satisfying 
$C_{8}p_{k}^{\frac{n(n+2) }{n+1}} \ge   C_{7}p_{k}^{\frac{n(n+2)}{n+1}+2}/(2C_{6}p_k(p_k+1))$ and 
define
\[
M:=\max\{  \norm{u_{0}}_{L^{1}(\Omega)}, \norm{u_{0}}_{L^{\infty}(\Omega)},  1\},\qquad\delta_{k}:=C_{8}p_{k}^{\frac{n(n+2) }{n+1}}.
\]
Then, it follows from \eqref{MA_2} that
\[
 \sup_{t\le T}\int_{\Omega}u^{p_k}(\cdot,t)\le \max\bket{  M^{p_k},\quad   \delta_{k}\biggr{(} \sup_{t\le T}\int_{\Omega}u^{p_{k-1}}(\cdot,t)\biggr{)}^{2}}.
\]
 Note that   $\delta_k\ge1$ for all $k=1,2,3,\cdots$. 
By an inductive computation, we have
 \[
  \sup_{t\le T}\biggr{(}\int_{\Omega}u^{p_k}(\cdot,t)\biggr{)}^{\frac{1}{p_{k}}}\le \bigr{[}\delta_{k}\delta_{k-1}^{p_{1}}\delta_{k-2}^{p_{2}}\cdots\delta_{1}^{p_{k-1}}  (1+\sup_{t\le T}\norm{u(\cdot,t)}_{L^{1}(\Omega)})^{p_{k}}  \bigr{]}^{\frac{1}{p_{k}}}M 
 \]
 \begin{equation}\label{MA_3}
 \le   C_{8}^{\sum_{i=1}^{k}\frac{1}{2^{i}}}2^{\frac{n(n+2)}{n+1}\sum_{i=1}^{k}\frac{i}{2^{i}}}(1+\sup_{t\le T}\norm{u(\cdot,t)}_{L^{1}(\Omega)}) M.
\end{equation}
Thus, by \eqref{UL1} and $\sum_{i=1}^{k}\frac{i}{2^{i}}<\infty$, taking the limit $k\rightarrow\infty$, we can obtain  \eqref{ULINFTY}.

Similarly, we can see from $\eqref{MODEL00}_{2}$  that
\[
\begin{aligned}
\frac{1}{p}\frac{d}{dt}\int_{\Omega}v^{p}&+\frac{4(p-1)}{p^{2}}\int_{\Omega}|\nabla v^{\frac{p}{2}}|^{2}+\frac{\mu}{k}\int_{\Omega} v^{p}(u+v)\\
&=m\int_{\Omega}uv^{p-1}+\mu\int_{\Omega}v^{p}+(p-1)\int_{\Omega}v^{p-1}\nabla v \,\mathcal{S}[u,v].
\end{aligned}
\]
Using Young's inequality and \eqref{SLINF}, we note that
\[
(p-1)\int_{\Omega}v^{p-1}\nabla v \,\mathcal{S}[u,v]\le \frac{p-1}{p^{2}}\int_{\Omega}|\nabla v^{\frac{p}{2}}|^{2}+C(p-1)\int_{\Omega}v^{p}.
\]
Analogously as above, using \eqref{ULINFTY}, we have
\begin{align}\label{MA_33}
\frac{1}{p}\frac{d}{dt}\int_{\Omega}v^{p}+\frac{3(p-1)}{p^{2}}\int_{\Omega}|\nabla v^{\frac{p}{2}}|^{2}\le C(p+1)\int_{\Omega}v^{p}+C,
\end{align}
where $C>0$ is a constant independent of $p$.  Now,  likewise \eqref{MA_2}, we can find positive constants $C_{9}$, $C_{10}$, and $C_{11}$ independent of $p=p_{k}=2^{k}$, $k=1,2,3,\cdots$ such that
\begin{equation}\label{MA_4}
 \frac{d}{dt}\bke{\int_{\Omega}v^{p_k}+C_{9}}+2C_{10}p_k(p_k+1)\bke{\int_{\Omega}v^{p_k}+C_{9}}\le C_{11}p_{k}^{\frac{n(n+2)}{n+1}+2}\bke{  \int_{\Omega} v^{ p_{k-1}} +C_{9} }^{2}.
\end{equation}
We take  a sufficiently large constant $C_{12}\ge1$  independent of $p_{k}$ satisfying 
\[C_{12}p_{k}^{\frac{n(n+2) }{n+1}} \ge   C_{11}p_{k}^{\frac{n(n+2)}{n+1}+2}/(2C_{10}p_k(p_k+1))\] and 
define
\[
M:=\max\{  \norm{v_{0}}_{L^{1}(\Omega)}, \norm{v_{0}}_{L^{\infty}(\Omega)}, C_{9}+1 \},\qquad\delta_{k}:=C_{12}p_{k}^{\frac{n(n+2) }{n+1}}.
\]Then, we obtain from \eqref{MA_4} that
\[
 \sup_{t\le T}\bke{\int_{\Omega}v^{p_k}(\cdot,t)+C_{9}}\le \max\bket{  (2M)^{p_k},\quad   \delta_{k}\sup_{t\le T}\biggr{(} \int_{\Omega}v^{p_{k-1}}(\cdot,t)+C_{9}\biggr{)}^{2}}.
\]
 Note that   $\delta_k\ge1$ for all $k=1,2,3,\cdots$. 
By an inductive computation, we have
 \[
  \sup_{t\le T}\biggr{(}\int_{\Omega}v^{p_k}(\cdot,t)+C_{9}\biggr{)}^{\frac{1}{p_{k}}}
  \]
  \[
  \le \bigr{[}\delta_{k}\delta_{k-1}^{p_{1}}\delta_{k-2}^{p_{2}}\cdots\delta_{1}^{p_{k-1}}  (\sup_{t\le T}\norm{v(\cdot,t)}_{L^{1}(\Omega)}+C_{9}+1)^{p_{k}}  \bigr{]}^{\frac{1}{p_{k}}} 2M 
 \]
\[
 \le   C_{12}^{\sum_{i=1}^{k}\frac{1}{2^{i}}}2^{\frac{n(n+2)}{n+1}\sum_{i=1}^{k}\frac{i}{2^{i}}}(\sup_{t\le T}\norm{v(\cdot,t)}_{L^{1}(\Omega)}+C_{9}+1) 2M. 
\]
Analogously to \eqref{MA_3},  we can see that
 \[
 \sup_{t\le T}\biggr{(}\int_{\Omega}v^{p_k}(\cdot,t)\biggr{)}^{\frac{1}{p_{k}}} \le \sup_{t\le T}\biggr{(}\int_{\Omega}v^{p_k}(\cdot,t)+C_{9}\biggr{)}^{\frac{1}{p_{k}}} 
\]
\[
 \le   C_{12}^{\sum_{i=1}^{k}\frac{1}{2^{i}}}2^{\frac{n(n+2)}{n+1}\sum_{i=1}^{k}\frac{i}{2^{i}}}(\sup_{t\le T}\norm{v(\cdot,t)}_{L^{1}(\Omega)}+C_{9}+1) 2M.
 \]
Due to \eqref{VL1} and $\sum_{i=1}^{k}\frac{i}{2^{i}}<\infty$, we have  \eqref{VLINFTY} by  taking the limit $k\rightarrow\infty$.  This completes the proof.
 
\end{proof}
\end{subsection}
\begin{pfthm1}
It is a direct consequence of local-in-time existence, uniqueness, non-negativity (Lemma~\ref{LEM1}), blow-up criterion (Lemma~\ref{BUCLEM}), and {\it a priori} estimates (Lemma~\ref{LEMLINFUV}). This completes the proof.
\end{pfthm1}

\section{Independent case}\label{SEC3}
In this section, we prove Theorem~\ref{GETHM1}.

 \begin{subsection}{Preliminary}
In what it follows we take such an example of $\Om$ and $E(x)$ and show a similar lemma as Lemma \ref{LEMKS} holds for ${\calK}$ and ${\calS}$.
\begin{lemma}\label{LEMKS1}
Let $\Omega$ be the open ball of radius $L$, $B_{L}(0)$, and let $0<R<L$. Suppose that $f,g \in W^{1,p}(\Omega)$, $p>n$. Assume further that $\mathcal{K}[f,g]$, $\mathcal{S}[f,g]$, $E$, and $\omega$ are given by \eqref{MODEL01}, and \eqref{sensing2}--\eqref{omega2}.
Then, there exists a  constant $C>0$ satisfying
\[ \norm{\mathcal{K}[f,g]}_{\iom}+\norm{\mathcal{S}[f,g]}_{\iom}\le C(R,\Omega)(\norm{f}_{L^1(\Omega)}+ \norm{g}_{L^1(\Omega)}),\]
\[\norm{\mathcal{K}[f,g]}_{W^{1,\infty}(\Omega)}+\norm{\mathcal{S}[f,g]}_{W^{1, \infty}(\Omega)}\le C(R,\Omega)(\norm{f}_{W^{1,p}(\Omega)}+ \norm{g}_{W^{1,p}(\Omega)}). 
\]
\end{lemma}
\begin{proof}
It is sufficient to show 
\begin{align}
\label{lemks1eq}
\| \mathcal{I}[f]\|_{L^{\infty}(\Omega)} &\le C(R,\Om)\| f\|_{L^1(\Omega)}, \\
\label{lemks1eq0}
\| \mathcal{I}[f]\|_{W^{1,\infty}(\Omega)} 
&\le C(R,\Om)\| f\|_{W^{1,p}(\Omega)}, 
\end{align} where
\[
\mathcal{I}[f](x)=\displaystyle\int_{E(x)} f(x+\br)      \omega(\br)d\br.
\]
\eqref{lemks1eq} is obvious. 
We prove \eqref{lemks1eq0} when $B_L(0) $ is the two dimensional disc in  $\bbr^2$ for a computational simplicity.  
Let us  denote the radial coordinate of $B_L(0)$ by $(s, \vp)$, thus consider $ x=(x_1, x_2)= (s\cos \vp, s\sin \vp)$. We may assume $f$ is in $C^1(B_L(0))$  since $ f\in W^{1.p}(\Om)$ has an approximating sequences from $C^{\infty}(\Om)$ in $W^{1. p}(\Om)$.\\
\indent
When $ |x| <L-R$, we have
\begin{equation}\label{lemks1eq1}
\pa_{x_1} \calI[f](x) = \int_{E(x)} f_{x_1}(x+\br) \om(\br) d\br:= I^{in}(x).
\end{equation}
When $ L-R<|x| <L$, we use  polar coordinates $(r, \theta)$ on  $ B_{L-s}(0)$ to write
\begin{align*}
\calI[f](x)= \int_0^{L-s} \int_0^{2\pi} f(x+\br) \om(\br) r d\theta dr.
\end{align*}
In this region we have
\begin{equation}\label{lemkseq2}
\begin{aligned} 
\pa_{x_1} \calI[f](x)=& \frac{\pa s}{\pa x_1}  \dds \int_0^{L-s} \int_0^{2\pi} f(x+\br) \om(\br) r d\theta dr +
 \frac{\pa \vp}{\pa x_1}\ddp \int_0^{L-s} \int_0^{2\pi} f(x+\br) \om(\br) r d\theta dr\\
= &-  \frac{\pa s}{\pa x_1}\int_0^{2\pi} f(x + (L-s)(\cos \theta, \sin \theta) )\om ((L-s)(\cos\theta, \sin\theta)) d\theta \\
&+ \int_{E(x)} f_{x_1}(x+\br)\om(\br) d\br\\
 :=& I^o(x)
\end{aligned}
\end{equation}
with $ \frac{\pa s}{\pa x_1}= \cos \vp$.
From the above we have
\begin{align*}
\| \pa_{x_1} \calI[f]\|_{L^{\infty}(\Om)} \le  C\|f\|_{L^{\infty}(\Om)} \| \om\|_{L^{\infty}(\Om)} 
+C \|\om \|_{L^{p'}(\Om)} \| f_{x_1}\|_{L^p(\Om)} 
\end{align*} for $\frac{1}{ p'} + \frac{1}{p} =1$.
The same bound holds for $\| \pa_{x_2} \calI[f]\|_{L^{\infty}(\Om)}$.
 Hence we prove 
\begin{equation}\label{lemks1eq3}
\| \pa_{x} \calI[f]\|_{L^{\infty}(\Om)} \le C(R, \Om) \|f\|_{W^{1,p}(\Om)},
\end{equation}
where $\pa_x\calI[f]$ denotes the pointwise differentiation as above.\\
\indent 
For $\psi \in  C_0^{\infty}(\Om)$ we have
\begin{align*}
\int_{\Om} \calI[f](x) \pa_{x_1} \psi(x) dx 
&=\lim_{\ep\to 0}\int_{|x|< L-R-\ep}  \calI[f](x) \pa_{x_1} \psi(x) dx\\
&\quad + 
 \lim_{\ep\to 0}\int_{L-R + \ep <|x|< L}  \calI[f](x) \pa_{x_1} \psi(x) dx\\ 
& = -\int_{|x|< L-R }I^{in}(x)\psi(x) dx 
- \int_{L-R <|x|< L } I^o(x)\psi(x) dx\\
&\quad + \lim_{\ep\to 0} \int_{|x|= L-R-\ep} \calI[f](x) \psi(x) \frac{x_1}{L-R-\ep} dx\\
&\quad -\lim_{\ep\to 0} \int_{|x|= L-R+\ep} \calI[f](x) \psi(x) \frac{x_1}{L-R+\ep} dx\\
& =
 -\int_{|x|< L-R }I^{in}(x)\psi(x) dx 
- \int_{L-R <|x|< L } I^o(x)\psi(x) dx 
\end{align*}
since $ \calI[f]$ is continuous on $ |x|= L-R$. Hence
\[ D_{x_1} \calI[f](x): = \begin{cases} 
I^{in}(x)  \quad  |x| < L-R\\
I^o(x)  \quad   L-R< |x| < L
\end{cases}\]
 is the weak derivative of $\calI[f]$ in ${x_1}$  and \eqref{lemks1eq0} follows due to \eqref{lemks1eq3}. This completes the proof.
\end{proof}

\end{subsection}
\begin{subsection}{Local well-posedness, blow-up criteria, and a priori estimates}
To prove  local well-posedness of \eqref{MODEL11}--\eqref{omega2}, we introduce the sectorial operator $A:=(I-\Delta)|_{D(A)}$ with its domain
\[
D(A):=\{  f\in W^{2,p}(\Omega)\,|\,Bf:=\partial_{\nu}f=0\,\mbox{ on }\,\partial\Omega   \},\quad 1<p<\infty.
\]
Since $(A,B)$ is a sectorial operator, it generates an analytic semigroup $\{ e^{-tA}\,|\,t\ge0  \}$ on $L^{p}(\Omega)$. 
Note   that the fractional powers of $A$ are well-defined (see \cite[Section 1.4]{H81}). We denote the domains of fractional powers by
\[
X^{\eta}_{p}:=D(A^{\eta}),\quad  \eta\in(0,1).
\]
We also note from  \cite[Theorem 1.6.1]{H81} that 
\begin{equation}\label{FRACEM}
X^{\eta}_{p}\hookrightarrow W^{\kappa,q}(\Omega)\,\mbox{ for }\,\kappa-\frac{n}{q}<2\eta-\frac{n}{p}, \,q\ge p.
\end{equation}
Now we consider the homogenous Neumann boundary value problems for   \eqref{MODEL11}--\eqref{omega2}: 
\begin{equation}\label{NPBV1}
\begin{aligned}
&\partial_{t}u+(I-\Delta)u=g_{1},&  &\partial_{t}v+(I-\Delta)v=g_{2},   \quad\qquad&x\in\Omega, \,\,t\le T,\,\,\,&\\
&\partial_{\nu}u=  0, &\qquad\quad  &\partial_{\nu} v=0,  &x\in\partial\Omega, \,\,t\le T,&\\
&u(x,0) =u_{0}(x),&\qquad  &v(x,0) =v_{0}(x),  &x\in\partial\Omega,\quad\qquad\,&
\end{aligned}
\end{equation}
and their integral representation formulas: 
 \begin{equation}\label{INTF11}
u=e^{-tA }u_{0}+\int_{0}^{t}e^{-(t-\tau)A } g_{1}(\tau) d\tau,\qquad v=e^{-tA }v_{0}+\int_{0}^{t}e^{-(t-\tau)A} g_{2}(\tau) d\tau,
\end{equation}
where   
 \begin{equation}\label{INTF12}
  \begin{array}{ll}
g_{1}:=-\nabla \cdot (u \,\mathcal{K}[u,v]     )+(1-m)u +\displaystyle\frac{\lambda}{k} u(k-(u+v)),
\\
g_{2}:=-\nabla \cdot (v \,\mathcal{S}[u,v]     )+mu+v +\displaystyle\frac{\mu}{k} v(k-(u+v)).
  \end{array}
\end{equation}
Before stating the local-in-time result, let us note the smoothing estimates   for $e^{-t A} $ from  \cite[Theorem 1.4.3]{H81}, which will be used in the proof of the next lemma:\\

If $p>1$, $0<\eta<1$, $f\in L^{p}(\Omega)$, then there exist positive constants   $\kappa$, and $C=C(\eta)$  such that
\begin{equation}\label{NMNEMB}
\norm{ e^{-tA}f}_{X^{\eta}_{p}}\le Ce^{-\kappa t}t^{-\eta}\norm{f}_{L^{p}(\Omega)},\quad t>0.
\end{equation}

\begin{lemma}\label{LEM4}
Let $\Omega \subset\R^{n}$, $n\ge2$,  be an open ball of radius $L$, $B_{L}(0)$. Assume that   $u_{0}$ and $v_{0}$  are  non-negative functions belong to $W^{2,p}(\Omega)$, $p>n$, and satisfy compatibility condition
\[
\partial_{\nu}u_0=0,\quad \partial_{\nu}v_0=0,\qquad x\in\partial\Omega.
\]  
Then, there is a maximal time, $T_{\rm max}\le\infty$, such that a pair of unique non-negative strong solution $(u,v)$ of \eqref{NPBV1} exists  and satisfies
\[
u,v\in C([0,T_{\rm max}); W^{1,p}(\Omega))\cap W^{1,p}(0,t; L^{p}(\Omega))\cap L^{p}(0,t;W^{2,p}(\Omega)),\qquad t<T_{\rm max}. 
\]
Moreover, if $T_{\rm max}<\infty$, then
\begin{equation}\label{BUC1}
\lim_{t\rightarrow T_{\rm max}} (\norm{u(\cdot,t)}_{W^{1,p}(\Omega)}+\norm{v(\cdot,t)}_{W^{1,p}(\Omega)} )=\infty.
\end{equation}
\end{lemma}
\begin{proof}
Let $n<p<\infty$. With  positive constants $T<1$ and $R_0$ to be specified below, we introduce the Banach space $X_{T}:=C([0,T];W^{1,p}(\Omega))$ and its closed convex subset $S_{T}\subset X_{T}$,
\[
S_{T}:=\bigr{\{}\,f\in X_{T}\,|\, \norm{f}_{L^{\infty}(0,T; W^{1,p}(\Omega))}\le R_0   \bigr{\}}.
\]
Let $u,v \in S_{T}$.  As in \eqref{INTF11}--\eqref{INTF12}, we consider \begin{equation}\label{PFPHI11}
\Phi_{1}(u,v):=e^{-tA }u_{0}+\int_{0}^{t}e^{-(t-\tau)A }g_{1}(\tau) d\tau,
\end{equation}
\begin{equation}\label{PFPHI22}
\Phi_{2}(u,v):=e^{-tA }v_{0}+\int_{0}^{t}e^{-(t-\tau)A } g_{2}(\tau) d\tau,
\end{equation}
where $g_{i}$ for $i=1,2$ are functions given in \eqref{INTF12}.

First, we show  $\Phi_{1}(u,v),\Phi_{2}(u,v)\in S_{T}$.    Let $\frac{1}{2}<\eta<1$ and note from \eqref{FRACEM} that
\begin{equation}\label{FRACEM1}
X^{\eta}_{p}\hookrightarrow W^{1,p}(\Omega).
\end{equation}
Using \eqref{NMNEMB} and \eqref{FRACEM1},
we can compute
\[
\begin{aligned}
\norm{\Phi_{1}(u,v)}_{W^{1,p}(\Omega)}&\le \norm{e^{-tA}u_{0}}_{W^{1,p}(\Omega)}+\int_{0}^{t}\|e^{-(t-\tau)A}g_{1}(\tau)\|_{W^{1,p}(\Omega)}d\tau\\
&
\le K_{3}\norm{u_{0}}_{W^{1,p}(\Omega)}+C\int_{0}^{t}\|e^{-(t-\tau)A}g_{1}(\tau)\|_{X^{\eta}_{p}}d\tau\\
&
\le  K_{3}\norm{u_{0}}_{W^{1,p}(\Omega)}+C\int_{0}^{t}e^{-\kappa(t-\tau)}(t-\tau)^{-\eta}\norm{g_{1}(\tau)}_{L^{p}(\Omega)} d\tau.
\end{aligned}
\]
By Lemma~\ref{LEMKS1} and $W^{1,p}(\Omega)\hookrightarrow L^{\infty}(\Omega)$, we have  
\[
\norm{g_{1}(\tau)}_{L^{p}(\Omega)}=\bigr{\|}-\nabla \cdot (u \,\mathcal{K}[u,v]     )+(1-m)u +\displaystyle\frac{\lambda}{k} u(k-(u+v))\bigr{\|}_{L^{p}(\Omega)}\le   C (R_0+R_0^{2}),
\]
and thus,
\[
\norm{\Phi_{1}(u,v)}_{W^{1,p}(\Omega)}\le K_{3}\norm{u_{0}}_{W^{1,p}(\Omega)}+ C_{13}(R_0+R_0^{2})t^{1-\eta},
\]
where $C_{13}$ is a positive constant independent of $R_0$.
Similarly, we  also have 
\[
\norm{\Phi_{2}(u,v)}_{W^{1,p}(\Omega)}\le K_{4}\norm{v_{0}}_{W^{1,p}(\Omega)}+ C_{14}(R_0+R_0^{2})t^{1-\eta},
\]
where  $C_{14}$ is a positive constant independent of $R_0$. 
Taking 
\[
R_0:=K_{3}\norm{u_{0}}_{W^{1,p}(\Omega)}+K_{4}\norm{v_{0}}_{W^{1,p}(\Omega)}+1,
\]
 and 
\[
T<T_{3}:=\min \bket{ 1, \frac{1}{(2 C_{13}(R_0+R_0^{2}))^{\frac{1}{ 1-\eta}}}, \frac{1}{(2 C_{14} (R_0+R_0^{2}))^{\frac{1}{1-\eta}}} },
\]
we obtain 
$
\Phi_{1}(u,v),\Phi_{2}(u,v)\in S_{T}$.

We next show the mapping $(u,v)\mapsto (\Phi_{1},\Phi_{2})$ is a contraction. Using \eqref{INTF11}, we note that
\[
\Phi_{1}(u,v)-\Phi_{1}(\tilde{u},\tilde{v})=\int_{0}^{t}e^{-(t-\tau)A}(g_{1}-\tilde{g}_{1})(\tau)d\tau,
\]
\[
\Phi_{2}(u,v)-\Phi_{2}(\tilde{u},\tilde{v})=\int_{0}^{t}e^{-(t-\tau)A}(g_{2}-\tilde{g}_{2})(\tau) d\tau,
\]
where 
\[
\tilde{g}_{1}=-\nabla \cdot (\tilde{u} \,\mathcal{K}[\tilde{u},\tilde{v}]     )+(1-m)\tilde{u} +\displaystyle\frac{\lambda}{k} \tilde{u}(k-(\tilde{u}+\tilde{v})), 
\]
\[
\tilde{g}_{2}=-\nabla \cdot (\tilde{v} \,\mathcal{S}[\tilde{u},\tilde{v}]     )+ m\tilde{u} +\tilde{v}+\displaystyle\frac{\mu}{k} \tilde{v}(k-(\tilde{u}+\tilde{v})).
\]
Then, analogously to the above  computations, we can estimate
\[
\norm{(\Phi_{1},\Phi_{2})(u,v)-(\Phi_{1},\Phi_{2})(\tilde{u},\tilde{v})}_{W^{1,p}(\Omega)}\le  C_{15} (R_0+1)T^{1-\eta}\norm{ (u,v)- (\tilde{u},\tilde{v})}_{W^{1,p}(\Omega)}.
\]
Taking 
\[
T<T_{4}:=\min \bket{T_{3},  \frac{1}{(2 C_{15} (R_0+1))^{\frac{1}{ 1-\eta}}} } ,
\] we obtain 
\[
\norm{(\Phi_{1},\Phi_{2})(u,v)-(\Phi_{1},\Phi_{2})(\tilde{u},\tilde{v})}_{W^{1,p}(\Omega)}\le \frac{1}{2}\norm{ (u,v)- (\tilde{u},\tilde{v})}_{W^{1,p}(\Omega)},
\]
i.e., the mapping is a contraction. According to the Banach fixed point theorem, this mapping has a fixed point in $S_{T}$, denoted again as $(u,v)$.

  The uniqueness of $(u,v)$, and the blow-up criterion \eqref{BUC1} can be obtained as in the proof of Lemma~\ref{LEM1}. By  the maximal regularity theorem   (see, e.g., Ladyzhenskaya-Solonnikov-Ural\'ceva \cite[Section 4. Theorem 9.1]{LSU88}), we have
\[
u,v\in L^{p}(0,T; W^{2,p}(\Omega))\cap W^{1,p}(0,T; L^{p}(\Omega)).
\]   
Then, the non-negativity of solutions is followed as \eqref{UNU}--\eqref{UNV}. This completes the proof.
\end{proof}
 
Next, we  provide a refined blow-up criterion (Lemma~\ref{BUCLEM1}), and  a priori estimates (Lemma~\ref{LEML11} and Lemma~\ref{LEMLINFUV1}). In order to reduce the redundancies, we refer to see the proofs of Lemma~\ref{BUCLEM},  Lemma~\ref{LEML1}, and Lemma~\ref{LEMLINFUV}. 

\begin{lemma}\label{BUCLEM1}
Let  the same assumptions as in Lemma~\ref{LEM4} be satisfied. The solution $(u,v)$ of \eqref{NPBV1} given by Lemma~\ref{LEM4} satisfies
\[
\mbox{either }\, T_{\rm max}=\infty,\,\mbox{ or }\, \lim_{t\rightarrow T_{\rm max}}(\norm{u(\cdot,t)}_{L^{\infty}(\Omega)}+\norm{v(\cdot,t)}_{L^{\infty}(\Omega)})=\infty.
\]
\end{lemma}
\begin{lemma}\label{LEML11}
Let the same assumptions as in Lemma~\ref{LEM4} be satisfied.  The solution  $(u,v)$ of   \eqref{NPBV1} given by Lemma~\ref{LEM4} for $T<T_{\rm max}$ satisfies
\begin{equation}\label{UL11}
\sup_{t\le T}\int_{\Omega}u(\cdot,t)\le C(\norm{u_{0}}_{L^{1}(\Omega)}),
\end{equation}
\begin{equation}\label{VL11}
\sup_{t\le T}\int_{\Omega}v(\cdot,t)\le C(\norm{u_{0}}_{L^{1}(\Omega)}, \norm{v_{0}}_{L^{1}(\Omega)}),
\end{equation}
\begin{equation}\label{KLINF1}
\sup_{t\le T} \norm{\mathcal{K}[u,v](\cdot, t)}_{L^{\infty}(\Omega)}\le C(\norm{u_{0}}_{L^{1}(\Omega)},  \norm{v_{0}}_{L^{1}(\Omega)}),
\end{equation}
and
\begin{equation}\label{SLINF1}
\sup_{t\le T} \norm{\mathcal{S}[u,v](\cdot, t)}_{L^{\infty}(\Omega)}\le C(\norm{u_{0}}_{L^{1}(\Omega)},  \norm{v_{0}}_{L^{1}(\Omega)}).
\end{equation}
\end{lemma}

\begin{lemma}\label{LEMLINFUV1}
Let the same assumptions as in Lemma~\ref{LEM4} be satisfied.  The solution  $(u,v)$ of   \eqref{NPBV1} given by Lemma~\ref{LEM4} for $T<T_{\rm max}$ satisfies
\begin{equation}\label{ULINFTY1}
\sup_{t\le T}\norm{u(\cdot,t)}_{L^{\infty}(\Omega)}\le  C(\norm{u_{0}}_{(L^{1}\cap L^{\infty})(\Omega)}, \norm{v_{0}}_{L^{1}(\Omega)}),
\end{equation}
and
\begin{equation}\label{VLINFTY1}
\sup_{t\le T}\norm{v(\cdot,t)}_{L^{\infty}(\Omega)}\le  C(\norm{u_{0}}_{(L^{1}\cap L^{\infty})(\Omega)}, \norm{v_{0}}_{(L^{1}\cap L^{\infty})(\Omega)}).
\end{equation}
\end{lemma}
\end{subsection}
\begin{pfthm2}
It is a direct consequence of   local-in-time existence, uniqueness, non-negativity (Lemma~\ref{LEM4}), the blow-up criterion (Lemma~\ref{BUCLEM1}), and a priori estimates (Lemma~\ref{LEMLINFUV1}). This completes the proof.
\end{pfthm2}

\begin{remark} 
If the adhesive strength of    $v$ is negligible, $\mathcal{S}[u,v]=0$, and  the growth rate of $u$ is smaller than that of $v$, $\lambda<\mu$, 
then the solution asymptotically tends to constant equilibrium. Indeed, the solution $(u,v)$ of    \eqref{MODEL11}--\eqref{omega2} with $\inf_{\overline{\Omega}}v_0>0$ given by Theorem~\ref{GETHM1}  satisfy
\[
\frac{d}{dt}\bkt{ \frac{1}{\lambda}\int_{\Omega}u +\frac{1}{\mu}\int_{\Omega}v -\frac{k}{\mu}\int_{\Omega}\log v   }
+\frac{k}{\mu}\int_{\Omega} \biggr{|}\frac{\nabla v}{v}  \biggr{|}^{2}+\frac{mk}{\mu}\int_{\Omega}\frac{u}{v}+m\bke{ \frac{1}{\lambda}-\frac{1}{\mu}  }\int_{\Omega}u+\frac{1}{k}\int_{\Omega}|u-(k-v)|^{2}
\]
\[
=\frac{k}{\mu}\int_{\Omega}\mathcal{S}[u,v]\frac{\nabla v}{v}.
\]
Thus, one can    verify that $\int_{\Omega}u(\cdot,t)\rightarrow 0$ and  $\int_{\Omega}|u-(k-v)|^2(\cdot,t)\rightarrow 0$ as time tends to infinity whenever $\mathcal{S}[u,v]=0$ and $\lambda<\mu$. Then, by     $u\ge0$, we have     $(u,v)\rightarrow(0,k)$.\end{remark}

%

\section*{Appendix}
Let $\Om, \calK, \calS$ and $\om$ be given as in  Section $3$. We define the operator 
$ F: W^{2,p}_B \times  W^{2,p}_B \to L^2\times L^2$ by 
\begin{align*}
F(u,v)&= (F_1(u,v), F_2(u,v)),\\
F_1(u,v)&= \Del u-\nabla \cdot (u \,\mathcal{K}[u,v]     )+(1-m)u +\displaystyle\frac{\lambda}{k} u(k-(u+v)) \\
F_2(u,v)&= \Del v-\nabla \cdot (v \,\mathcal{S}[u,v]     )+mu+v +\displaystyle\frac{\mu}{k} v(k-(u+v)).
\end{align*}
We denote the G\^ateaux derivative of $F$ at $U=(u,v)$ by $T_{U}$;\\
\[ T_{U}(W) = \lim_{t\to 0} \frac{F(U+tW)- F(U)}{t} = (\delta_W F_1(U), \delta_W F_2(U))\] where
$ W= (w, z)
 \in  W^{2,p}_B \times  W^{2,p}_B$.
 By computation we have
 \begin{align*}
 \delta_WF_1(0,k)& = \Del w-mw \\
 \delta_WF_2(0,k)& =  \Del z - \na \cdot( k\calS[w,0]+kS[w,z]+z\calS[0,k]) + (m-\mu)w - \mu z\\
 \delta_WF_1(0,0)& = \Del w - mw\\
 \delta_WF_2(0,0)& = \Del z + mw + \mu z.
 \end{align*}
 We consider the two linearized equations at $(0,k)$ and $(0,0)$ respectively with  initial data 
 $(w_0, z_0)\in W^{2,p}_B \times  W^{2,p}_B$;
 \begin{align}\label{L1}
 \begin{aligned}
 \pa_t w &=  \Del w-mw\\
 \pa_t z &= \Del z - \na \cdot( k\calS[w,0]+kS[w,z]+z\calS[0,k]) + (m-\mu)w - \mu z,
 \end{aligned}
 \end{align}
 and
  \begin{align}\label{L2}
 \begin{aligned}
 \pa_t w &=  \Del w-mw \\
 \pa_t z &=  \Del z + mw + \mu z. \phantom{ \na \cdot( k\calS[w,0]+kS[w,z]+z\calS[0,k]) + m+m}
 \end{aligned}
 \end{align}
 The equations \eqref{L1}, \eqref{L2} are decoupled and it is immediate that
 \begin{align}\label{apeq1}
 \| w\|_{W^{1,p}(\Om)} \le e^{-mt}\|w_0\|_{W^{1,p}}, \quad p\ge 1
 \end{align}
 from 
 \[ \pa_t(e^{m t} w ) = \Del (e^{m t} w ).\]
 Let $ \tz$ denote $e^{\mu t} z$.
 Multiplying to the $z$- equation of \eqref{L1} by ${e^{\mu t}}$, we have
 \begin{align}\label{apeq3}
 \pa_t \tz - \Del \tz = -\na \cdot( 2k\calS[e^{\mu t}w,0]+kS[0,\tz]+ \tz\calS[0,k]) + (m-\mu) e^{\mu t}w.
 \end{align} 
 It holds that
 \[ \ddt \into |\tz |\le  |m-\mu| e^{\mu t} \int |w| \le   |m-\mu| e^{-(m-\mu)t} \| w_0\|_{L^1(\Om)},\]
 which implies
 \begin{align*}
 \into |\tz |& \le \into |z_0| -   \frac{|m-\mu|}{m-\mu}(e^{-(m-\mu)t}-1)\into |w_0| 
 \end{align*}
 and 
 \begin{align}
  \into |z |& \le e^{-\mu t} \into |z_0| +   (e^{-mt}- e^{-\mu t})\into |w_0| \label{apeq111},
 \end{align}
 where we abuse the notation by 
 $ | m-\mu|/(m-\mu) =0$ if $ \mu=m$. When $m > \mu$,  it holds that 
 \begin{align}\label{apeq1111}
 \int | \tz| \le  \| z_0\|_{L^1(\Om)} + \| w_0\|_{L^1(\Om)}.
 \end{align}
 \indent 
In what follows we find that 
 the different signs of $\mp \mu z$ in \eqref{L1} and \eqref{L2} imply that $(0, k)$ is linearly stable and $(0,0)$ is linearly unstable as expected. 
 \begin{proposition}\label{stab}
  Each of the linearized equations  \eqref{L1}, \eqref{L2} have a unique global solution $(w,z)$ for each  in 
 \[C([0,t); W^{1,p}(\Omega))\cap W^{1,p}(0,t; L^{p}(\Omega))\cap L^{p}(0,t;W^{2,p}(\Omega))\]
 for any $t>0$. When $m> \mu$, the solution $(w,z)$ for \eqref{L1} is asymptotically stable such that
 \begin{align}\label{appeq2}
 \|z\|_{L^p(\Om) }\le e^{-\mu t}\| z_0\|_{L^p(\Om)} \quad \mbox{ for } p\ge 1.
 \end{align}
 The solution  $(w,z)$ for \eqref{L2} grows exponentially in its $L^1$-norm if the initial data is non-negative;
 \begin{align}\label{appeq3}
 \int_\Om |z| \ge e^{\mu t}\int_\Om |z_0|.
 \end{align}
 \end{proposition}
 \begin{proof}
 Due to the a priori estimates \eqref{apeq1} and \eqref{apeq111}, the global well-posedness part  for  \eqref{L1} follows from the same argument in the subsection $2.3$ or the subsection $3.2$. Repeating the argument of  Lemma \ref{LEMLINFUV} to \eqref{L1} it  holds that 
   \begin{equation}\label{MA-z}
   \| z\|_\iom \le C(\|w_0\|_{L^1\cap L^{\infty}(\Om)}, \|z_0\|_{L^1\cap L^{\infty}(\Om)}).
   \end{equation}
   For details see \eqref{tz}--\eqref{tzinf} for $\tz$, where the similar estimates are given.
   By Lemma \ref{BUCLEM1}  it also holds that 
   \begin{equation}\label{bw}
   \| z\|_{W^{1,p}} \le C(\|w_0\|_{L^1\cap L^{\infty}(\Om)}, \|z_0\|_{L^1\cap L^{\infty}(\Om)} )
   \end{equation}
   for any $p\ge 1$.
    Let us prove  \eqref{appeq3} first.  The solution $(w,z)$ remains non-negative  and we have
    \begin{align*}
    \into w& = e^{-mt}\into w_0,\\
    \ddt(e^{-\mu t} \into z) & = m e^{-(\mu+ m)t} \into w_0.
    \end{align*}
    Integrating the second equation, we have  \eqref{appeq3}. 
    \\
    \indent
    For \eqref{appeq2} we proceed as in Lemma \ref{LEMLINFUV}. Multiplying $|\tz|^{p-2}\tz$
    into \eqref{apeq3} for $p \ge2$, we have
    \begin{align}\label{tz}
    \frac 1p \ddt \into \tzp + \frac{4(p-1)}{p^2}\into |\na \tz^{\frac p2}|^2 = &
    \into |\tz|^{p-2}\tz \na \cdot( 2k\calS[e^{\mu t}w,0]+kS[0,\tz]+ \tz\calS[0,k])  \\
    & + (m-\mu) \into e^{\mu t}w |\tz|^{p-2}\tz.
   \end{align}
   By Lemma \ref{LEMKS1}, \eqref{apeq1},  \eqref{bw} and using $m >\mu$, we have
    \begin{align*}
   & \| \na \cdot \calS [ \emu w, 0]\|_\iom \le C\| \emu w\|_{W^{1,q}(\Om)} \le C \|w_0\|_{W^{1,q}(\Om)} (q >n)\\
   & \|  \calS[ 0, \tz] \|_\iom \le C\| \tz\|_{L^1(\Om)} 
 \le 
    C(\|w_0\|_{L^1}, \|z_0\|_{L^1})\\
  &  \| \na \calS[0, k]\|_\iom \le C
    \end{align*}
    and estimate the right hand side of \eqref{tz} as follows,
    \begin{align*}
   &\into |\tz|^{p-2}\tz \na \cdot( 2k\calS[e^{\mu t}w,0]) + (m-\mu) \into e^{\mu t}w |\tz|^{p-2}\tz  \le C\into \tzpp \\
    &\into |\tz|^{p-2}\tz \na \cdot kS[0,\tz]  \le \frac{p-1}{p^2} \into | \na \tz^{\frac p2|}|^2 + C(p-1) \int \tz^{p-2}
    \|  \calS[0, \tz\|_{\iom}^2 \\
  & \phantom{\into |\tz|^{p-2}\tz \na \cdot kS[0,\tz] } \le \frac{p-1}{p^2} \into | \na \tz^{\frac p2}|^2 + C \frac{p-1}{p} \left(|\Om| +(p-2)\into \tzp \right),\\
  & \into |\tz|^{p-2}\tz \na \cdot( \tz\calS[0,k]) \le \frac{1}{p^2} \into  | \na \tz^{\frac p2}|^2  + C\into \tzp \| \na \calS[0, k]\|_{\iom}.
   \end{align*}
   Summing up, we have
   \begin{align*}
    \frac 1p \ddt \into \tzp + \frac{3(p-1)}{p^2}\into |\na \tz^{\frac p2}|^2  \le C+ C\into \tzp, \quad p\ge 2,
   \end{align*}
   where $C$ is a uniform constant depending on $\| w_0\|_{L^1(\Om)}$, $\|z_0\|_{L^1(\Om)}$, and given constants $\mu, m, k$
   etc..  As was derived from \eqref{MA_33} for $v$ in Lemma \ref{LEMLINFUV} it holds that
   \[\sup_ {0 <t \le T} \| \tz\|_{L^{p_k}(\Om)} \le C(\|z_0\|_{L^1(\Om)}, \|z_0\|_{\iom}) \sup_ {0 <t \le T}\left(\| \tz \|_{L^1(\Om)} 
   +C \right)\quad p_k= 2^k, k=0, 1\dots.\]
and 
\begin{align}\label{tzinf} \sup_ {0 <t \le T} \| \tz\|_\iom \le  C(\|w_0\|_{L^1\cap L^{\infty}(\Om)}, \|z_0\|_{L^1\cap L^{\infty}(\Om)}).
\end {align}
That implies \eqref{appeq2}.
\end{proof}

\section*{Acknowledgement}
Jaewook Ahn's work is supported by NRF-2018R1D1A1B07047465. Jihoon Lee's work is supported by SSTF-BA1701-05.
Myeongju Chae's work is supported by NRF-2018R1A1A 3A04079376.


\end{document}